\documentclass[reqno]{amsart}
\usepackage{CJK}
\usepackage{hyperref}
\usepackage{cite}
\usepackage{amsmath}
\usepackage{mathrsfs}
\usepackage{bm}
\usepackage{graphicx}
\usepackage{epsfig}
\usepackage{epstopdf}
\usepackage{amssymb}

\numberwithin{equation}{section}

\newtheorem{theorem}{Theorem}[section]
\newtheorem{lemma}[theorem]{Lemma}
\newtheorem{definition}[theorem]{Definition}

\newtheorem{remark}[theorem]{Remark}
\newtheorem{corollary}[theorem]{Corollary}

\allowdisplaybreaks

\makeatletter
\@namedef{subjclassname@2020}{%
	\textup{2020} Mathematics Subject Classification}
\makeatother

\newcommand{ \mint }{ {\int\hspace{-0.38cm}-}}

\begin{document}
	
	\title[\hfil Integral Harnack-type estimates and the rate of extinction] {Integral Harnack-type estimates and the rate of extinction for nonlocal doubly nonlinear  parabolic equations}
	
	\author[B. Shang, C. Zhang \hfil \hfilneg]
	{Bin Shang, Chao Zhang$^*$}
	
	\thanks{$^*$Corresponding author.}
	
	\address{Bin Shang \hfill\break
		School of Mathematics and Statistics, Northeast Normal University,
		Changchun 130024,  P.R. China
	} \email{shangbin0521@163.com}
	
	\address{Chao Zhang\hfill\break
		School of Mathematics and Institute for Advanced Study in Mathematics, Harbin Institute of Technology,
		Harbin 150001, P.R. China} \email{czhangmath@hit.edu.cn}
	\maketitle

	\begin{abstract}
We establish local boundedness and integral Harnack-type inequalities for locally nonnegative weak solutions to nonlocal doubly nonlinear parabolic equations in the fast diffusion range by combining refined energy estimates with a careful treatment of the nonlocal terms. As an application, quantitative decay estimates for both the local mass and the local supremum of weak solutions are derived near the extinction time. Furthermore, we study the extinction phenomenon for bounded nonnegative weak solutions to the Cauchy--Dirichlet problem and obtain quantitative upper bounds for the extinction time.
	\end{abstract}

	\section{Introduction}
	\label{sec1}
	\par
	
	Let $E\subset \mathbb{R}^N$ be an open bounded set, $T>0$, and let $E_T:=E\times(0,T)$. In this paper, we study integral Harnack-type estimates for locally nonnegative weak solutions to the following nonlocal doubly nonlinear parabolic equation 
	\begin{align}
		\label{1.1}
		\partial_t \left(|u|^{q-1}u\right)+\mathcal{L}u=0 \quad \text{in } E_T,
	\end{align}
	where $p>1$, $q>0$ and the operator $\mathcal{L}$ is given by
	\begin{align*}
		\mathcal{L} u(x,t)=\mathrm{P.V.} \int_{\mathbb{R}^N} K(x,y,t)|u(x,t)-u(y,t)|^{p-2}(u(x,t)-u(y,t))\,dy.
	\end{align*}
	Here, P.V. means the Cauchy principal value and $K:\mathbb{R}^N\times\mathbb{R}^N\times (0,T]\rightarrow [0,\infty)$ is a symmetric measurable kernel satisfying
	\begin{align}
		\label{kernel}
		\frac{\Lambda^{-1}}{|x-y|^{N+sp}} \leq K(x,y,t)\equiv K(y,x,t) \leq \frac{\Lambda}{|x-y|^{N+sp}}\quad \mathrm{a.e. }\ x,y\in\mathbb{R}^N
	\end{align}
	for some $\Lambda \geq 1$ and $s \in(0,1)$. 
	
 Eq. \eqref{1.1} can be regarded as a nonlocal counterpart of the doubly nonlinear parabolic equation
\begin{align}
 	\label{1.2}
 	\partial_t \left(|u|^{q-1}u\right)-\mathrm{div}(|\nabla u|^{p-2}\nabla u)=0 \quad \text{in } E_T,
 \end{align}
which has attracted considerable attention owing to both its mathematical structure and its relevance to various physical phenomena, including shallow water flows \cite{ASD08} and glacier dynamics \cite{M76}, see also \cite{V06, LM18} for further classical applications. It is well known that Eq. \eqref{1.2} encompasses Trudinger's equation $(q=p-1)$, the evolutionary $p$-Laplace equation $(q=1)$, and the porous medium equation $(p=2)$ as special cases. There have been fruitful achievements on Eq. \eqref{1.2}, such as the existence of solutions \cite{BDMS18, MN23}, H\"{o}lder regularity \cite{BDL21, V92,S26}, higher integrability \cite{BDKS20, MSS23}, and Harnack type estimates \cite{V94,BDG23, BHS21}. We refer the reader to \cite{ GV06, KK07, MN23-1, M23} and the references therein for more related results.

For equation \eqref{1.1}, Banerjee-Garain-Kinnunen \cite{BGK23} investigated the pointwise behavior of weak solutions for $p>1$ and $q>0$. Within the same general framework, Misawa-Yamaura \cite{MY25} proved an expansion of positivity for nonnegative weak supersolutions, providing a fundamental tool for the qualitative analysis of the equation. Li \cite{L25} subsequently combined intrinsic scaling with a nonlocal De Giorgi argument to show that locally bounded sign-changing weak solutions are locally H\"{o}lder continuous in the doubly degenerate regime $p>2$ and $0<q<p-1$. The mixed singular-degenerate range was later addressed by Adimurthi and Modasiya \cite{AM25}, who established local H\"{o}lder continuity for bounded solutions under suitable restrictions on the parameters.  More recently, Savchenko-Skrypnik-Yevgenieva employed the expansion of positivity to establish Harnack inequalities for weak solutions in the upper fast diffusion regime \cite{SSY26} and the slow diffusion regime \cite{SSY261}. Related results for the case $q=p-1$ and the case $q=1$
with $p>2$ can be found in \cite{CN26} and \cite{P26,L26}, respectively. In addition, for $p=2$ and $q>1$, De Filippis \cite{DF26} proved local boundedness under a time-integrated tail condition. 
Building on this result, Cassanello-De Filippis-Stanko \cite{CDS26} derived an integral
Harnack inequality and further obtained the expansion of positivity.

In this paper, we are interested in studying $L^q$--$L^q$ Harnack-type estimates, which formally resemble $L^1$--$L^1$ Harnack-type estimates applied to $u^q$. Such estimates quantify the propagation of local integral quantities of weak solutions by relating their supremum over a time interval in a given ball to the corresponding infimum on a larger spatial set. Such integral Harnack-type estimates have been used to study the Cauchy problem for the fast diffusion equation \cite{HP85}, establish the existence of initial traces for weak solutions to singular parabolic $p$-Laplace equations \cite{DH90}, and analyze the extinction behavior of solutions to both classes of equations \cite{DK92}. They have also played a key role in deriving Harnack
inequalities for quasilinear parabolic equations of singular $p$-Laplace type \cite{DGV10} and for doubly nonlinear parabolic equations \cite{BDG23}. A systematic treatment of the corresponding theory can be found in the monograph \cite{DGV12}, while further developments for singular parabolic equations in the subcritical regime can be found in \cite{BIV10}. In the nonlocal setting, Cassanello-Ciani-Iannizzotto \cite{CCI26} established integral Harnack estimates and considered the extinction phenomenon for the fractional $p$-Laplace evolution equation, corresponding to $q=1$ and $1<p<2$, with tails defined by an essential supremum over time.

We also investigate the extinction phenomenon for the Cauchy--Dirichlet problem
\begin{align}
	\label{1.3}
	\begin{cases} 
		\partial_t \left(|u|^{q-1}u\right)+ \mathcal{L} u = 0 & \text{in } E_T \\ 
		u = 0 & \text{in } \mathbb{R}^N\backslash E\times (0, T) \\ 
		u(\cdot, 0) = u_0 & \text{in } \mathbb{R}^N, 
	\end{cases}
\end{align}
where $E$ is a bounded domain and the initial datum $u_0\in W_{0}^{s,p}(E)\cap L^{q+1}(E)$. The extinction phenomenon, namely the vanishing of solutions in finite time, is a characteristic feature of fast diffusion. In the nonlocal framework, extinction in finite time was established for the fractional porous medium equation in the Dirichlet setting in \cite{KL11} and for the corresponding Cauchy problem in \cite{PQR12}. Related results for the fractional $p$-Laplacian evolution equation posed in $\mathbb{R}^N$ were obtained in \cite{V21,V22} by means of suitable comparison principles. Moreover, regularization and decay estimates, as well as finite-time extinction, were analyzed in \cite{CH21} for a doubly nonlinear nonlocal evolution problem driven by the fractional $p$-Laplacian with homogeneous Dirichlet boundary conditions.

 
Motivated by these developments, we establish integral Harnack-type estimates involving nonlocal tail terms integrated in time for locally nonnegative weak solutions to \eqref{1.1} in the fast diffusion range. The main difficulty is to choose appropriate test functions and truncation levels that account for the nonlinear time derivative while controlling the nonlocal contributions via tail terms, thereby obtaining the energy estimates needed for the iteration and the derivation of integral Harnack inequalities. We first prove qualitative local boundedness under suitable conditions and then derive quantitative local bounds by combining refined energy estimates with a parabolic Sobolev-type embedding and a De Giorgi iteration. A key ingredient is the use of truncation levels that depend on time to absorb the tail in energy estimates. We then establish an $L^q$--$L^q$ Harnack-type inequality and combine it with the local boundedness estimates to obtain an $L^q$--$L^\infty$ counterpart.
Finally, we apply the integral Harnack estimates to study the decay of solutions as they approach the extinction time. We also investigate the finite extinction time for the associated
Cauchy--Dirichlet problem by using differential inequalities.

	\smallskip
	To state the definition of weak solutions, we denote by $W_0^{s,p}(E)$ the Sobolev space with zero boundary values, namely 
	\begin{align*}
		W_0^{s,p}(E):=\left\{u \in W^{s,p}(\mathbb{R}^N): u=0 \text { in } \mathbb{R}^N \backslash E \right\}.
	\end{align*}
	Moreover, we introduce the tail space
	\begin{align*}
		L_\alpha^m(\mathbb{R}^{N}):=\left\{v \in L_{\rm{loc}}^m(\mathbb{R}^{N}): \int_{\mathbb{R}^N} \frac{|v(x)|^m}{1+|x|^{N+\alpha}}\,dx<+\infty\right\}, \quad m>0 \text{ and } \alpha>0.
	\end{align*}
	The parabolic nonlocal tail is of the form
	\begin{equation}
		\label{1.4}
		\mathcal T_{\pm}(u;x_0,R;t_1,t_2)
		:=\int_{t_1}^{t_2}\int_{\mathbb R^N\setminus K_R(x_0)}
		\frac{u_{\pm}(y,\tau)^{p-1}}{|y-x_0|^{N+sp}}\,dyd\tau.
	\end{equation}
	Note that $\mathcal T_{\pm}(u;x_0,R;t_1,t_2)$ is well-defined for any $u\in L^{p-1}(t_1,t_2;L_{sp}^{p-1}(\mathbb{R}^N))$.
	
	\smallskip
	
	We now state the notion of weak solutions to equation \eqref{1.1}.
\begin{definition}
	We say that a function $u$ is a local weak subsolution (super-) to the doubly nonlinear nonlocal parabolic equation \eqref{1.1} if 
	\begin{align*}
		u\in C_{\mathrm{loc}}(0,T; L^{q+1}_{\mathrm{loc}}(E))\cap L^p_{\mathrm{loc}}(0,T;W^{s,p}_{\mathrm{loc}}(E))\cap L^{p-1}_{\mathrm{loc}}(0,T;L_{sp}^{p-1}(\mathbb{R}^N)),
	\end{align*}
	and for every $E'\Subset E$, every $0<t_1<t_2<T$, there holds
	\begin{align}
		\label{1.5}
		-\int_{t_1}^{t_2}\int_{E'}|u|^{q-1}u\,\partial_\tau\varphi\,dxd\tau+\int_{t_1}^{t_2}\mathcal E(u,\varphi,\tau)\,d\tau
		\leq (\geq)\,0,
	\end{align}
	for all nonnegative test functions $\varphi \in W_0^{1, q+1}(t_1,t_2;L^{q+1}(E'))\cap L^p(t_1,t_2;W_0^{s,p}(E'))$, where
\begin{align*}
	\mathcal E(u,\varphi,\tau)
	:=\frac12\int_{\mathbb R^N}\int_{\mathbb R^N}
	&K(x,y,\tau)|u(x,\tau)-u(y,\tau)|^{p-2}
	(u(x,\tau)-u(y,\tau))\\
	&\times(\varphi(x,\tau)-\varphi(y,\tau))\,dydx.
\end{align*}
	A function $u$ is called a weak solution to \eqref{1.1} if it is both a weak subsolution and a weak supersolution.
\end{definition}
	The notion of weak solutions to the Cauchy--Dirichlet problem \eqref{1.3} is defined as follows.
	\begin{definition}
		A function
		\begin{align*}
			u\in C([0,T]; L^{q+1}(E))\cap L^p(0,T;W_0^{s,p}(E))
			\end{align*}
		 is called a weak solution to \eqref{1.3} if 
			\begin{align*}
			\iint_{E_T}(|u_0|^{q-1}u_0-|u|^{q-1}u) \partial_\tau \varphi\,dxd\tau+\int_{0}^{T} \mathcal{E}(u, \varphi, \tau)\,d\tau=0,
			\end{align*}
			for all $\varphi \in W^{1, q+1}(0,T;L^{q+1}(E))\cap L^p(0,T;W_0^{s,p}(E))$ such that $\varphi(\cdot,T)=0$.
		\end{definition}
	
	At this stage, we state our main results. The first one concerns the local boundedness of weak subsolutions.
	
	\begin{theorem}
		\label{thm-1-3}
		Let $0<p-1<q$ and let $r\geq 1$ be such that 
		\begin{align}
			\label{1.6}
			\lambda_r:=N(p-q-1)+rsp>0.
		\end{align}
		Suppose that $u$ is a weak subsolution to \eqref{1.1} satisfying $u\geq0$ in $Q_{\rho,\theta}:=K_\rho(x_0)\times(t_0-\theta,t_0]	\Subset E_T$. For $r>m:=p\frac{N+s(q+1)}{N}$, we further assume that $u$ is qualitatively locally bounded. Then we have
		\begin{align}
			\label{1.7}
			\underset{Q_{\frac{1}{2} \rho, \frac{1}{2} \theta}}{\operatorname{ess} \sup } \ u \leq & \gamma\left(\frac{\rho^{sp}}{\theta}\right)^{\frac{N}{\lambda_r}}\left(\mint\!\!\mint_{Q_{\rho,\theta}} u^r\,dxd\tau\right)^{\frac{sp}{\lambda_r}}+\gamma\left(\frac{\theta}{\rho^{sp}}\right)^{\frac{1}{q+1-p}}\nonumber\\
			&+\gamma\mathcal T_+(u;x_0,\rho/2;t_0-\theta,t_0)^{\frac1q},
		\end{align}
		where $\gamma>0$ depends only on $N,p,s,q,\Lambda,r$.	
	\end{theorem}

	\begin{remark}
		\label{rem-1-4}
		We shall prove the local boundedness result by distinguishing three cases. The first two cases deal with the range $r\leq m$, in which $u\in L_{\rm{loc}}^m(E_T)$ follows from the Sobolev-type embedding in Lemma \ref{lem-2-2}. For $r>m$, we additionally assume that $u$ is locally bounded. This allows us to control the $L^r$ quantity in the iteration by
		using the $L^m$ estimate provided by Lemma~\ref{lem-2-2}. Notice that there exists some $r\geq 1$ satisfying $r\leq m$ and $\lambda_r>0$ if and only if $\lambda_m>0$. Owing to $$\lambda_m=(N+sp)(m-q-1)=\frac{N+sp}{N}\lambda_{q+1},$$ $\lambda_m>0$ is equivalent to $m>q+1$ or $\lambda_{q+1}>0$, namely
		\begin{align*}
			q<\frac{N(p-1)+sp}{(N-sp)_+}.
		\end{align*}
		Consequently, we can apply Theorem \ref{thm-1-3} with $r=q+1$ without assuming a priori  that $u$ is locally bounded in this case.
	\end{remark}
	
As a consequence of Remark \ref{rem-1-4}, we obtain the following corollary.
	\begin{corollary}
		\label{cor-1-5}
		Let $0<p-1<q<\frac{N(p-1)+sp}{(N-sp)_+}$. Then every nonnegative weak subsolution to \eqref{1.1} is locally bounded.
	\end{corollary}
	
	Building on the local boundedness result, we derive an integral-type Harnack inequality, which serves as a key ingredient in obtaining the decay estimate.
	
	\begin{theorem}
		\label{thm-1-6}
		Let $0<p-1<q$. Suppose that $u$ is a weak solution to \eqref{1.1} satisfying $u\geq 0$ in $K_{4\rho}(x_0)\times(t_0-\theta,t_0]\Subset E_T$. Then the following estimate holds
	\begin{align*}
		&\sup_{t_0-\theta<\tau<t_0}\int_{K_\rho(x_0)} u^q(x,\tau)\,dx 
		\leq \gamma\left[\inf _{t_0-\theta<\tau<t_0} \int_{K_{2\rho}(x_0)} u^q(x,\tau) \,dx\right]\\
		&\qquad\qquad\quad+\gamma \rho^N\left[\left(\frac{\theta}{\rho^{sp}}\right)^{\frac{q}{q+1-p}}+\mathcal{T}_+(u;x_0,\rho/2;t_0-\theta,t_0)+\mathcal{T}_-(u;x_0,\rho/2;t_0-\theta,t_0)\right], 
	\end{align*}
		where $\gamma>0$ depends only on $N, p, s, q, \Lambda$.
	\end{theorem}
	
	\begin{remark}
		\label{rem-5-1}
		If the terminal time is fixed, the argument in the proof of Theorem~\ref{thm-1-6}
		also gives an estimate involving only the tail of the negative part of $u$. More precisely, let
		$0<t_1<t_2$ and assume that $u\geq0$ in $K_{4\rho}(x_0)\times(t_1,t_2]\Subset E_T $. Then
		\begin{align}
			\label{1.9}
			&\sup_{t_1<\tau<t_2}\int_{K_\rho(x_0)}u^q(x,\tau)\,dx\leq\gamma\int_{K_{2\rho}(x_0)}u^q(x,t_2)\,dx\nonumber\\
			&\quad\quad\quad+\gamma\rho^N\left[\left(\frac{t_2-t_1}{\rho^{sp}}\right)^{\frac{q}{q+1-p}}+\mathcal{T}_-(u;x_0,\rho/2,t_1,t_2)\right].
		\end{align}
		Indeed, in this case only the argument corresponding to
		$t_1<t_2$ in the proof of Theorem \ref{thm-1-6} is needed,
	and hence no positive tail term occurs.
	\end{remark}

	Combining the above $L^q$--$L^q$ Harnack-type inequality with the local boundedness estimate in Theorem \ref{thm-1-3}, we obtain the following pointwise version.
\begin{theorem}
	\label{thm-1-7}
	Let $0<p-1<q$ and $\lambda_q:=N(p-1-q)+spq>0$. Suppose that $u$ is a weak solution to \eqref{1.1} satisfying $u\geq 0$ in $K_{4\rho}(x_0)\times(t_0-\theta,t_0]\Subset E_T$. Then there exists a constant $\gamma>0$ depending only on $N, p, s, q, \Lambda$ such that
	\begin{align*}
		\operatorname*{ess\,sup}_{Q_{\rho/2,\theta/2}}u
		&\leq\gamma\theta^{-\frac{N}{\lambda_q}}
		\left[\inf_{t_0-\theta<\tau<t_0}\int_{K_{2\rho}(x_0)}u^q(x,\tau)\,dx\right]^{\frac{sp}{\lambda_q}}\\
		&\quad+\gamma
		\left(\frac{\theta}{\rho^{sp}}\right)^{\frac1{q+1-p}}\max\left\{1,\left(\frac{\theta}{\rho^{sp}}\right)^{-\frac{q}{q+1-p}}	\mathcal T_+(u;x_0,\rho/2;t_0-\theta,t_0)
		\right\}^{\frac{sp}{\lambda_q}}\\
		&\quad+\gamma
		\left(\frac{\theta}{\rho^{sp}}\right)^{\frac1{q+1-p}}
		\max\left\{
		1,
		\left(\frac{\theta}{\rho^{sp}}\right)^{-\frac{q}{q+1-p}}
		\mathcal T_-(u;x_0,\rho/2;t_0-\theta,t_0)
		\right\}^{\frac{sp}{\lambda_q}}.
	\end{align*}
\end{theorem}

		
We next turn to the Cauchy--Dirichlet problem \eqref{1.3} and establish the finite-time extinction of globally nonnegative weak solutions.
	
	\begin{theorem}
		\label{thm-1-8}
		Let $E\subset \mathbb{R}^N$ be a bounded open set. Let
		$0<p-1<q$, $sp<N$ and $\lambda_{q+1}:=N(p-1-q)+sp(q+1)$. Suppose that $u_0\in W_0^{s,p}(E)\cap L^\infty(\mathbb{R}^N)$,
	$u_0\geq 0$ in $E$,
		and that $u\geq 0$ is a bounded weak solution to \eqref{1.3} on every finite time interval.
		Then there exists a finite extinction time $T_*\geq0$ such that $u(\cdot,t)=0$ a.e. in $\mathbb{R}^N$
		for all $t\geq T_*$. Define
		$$\alpha=\frac{N(q+1-p)}{sp}, \qquad p_\alpha=\frac{N(q+1)}{N+s(q+1)}.$$
		Moreover, there exists a constant
		$\gamma>0$ depending only on $N,p,s,q,\Lambda$ such that the following estimates hold,
		\begin{itemize}
			\item[(i)] if $1<p<p_\alpha$, then $T_*\leq\gamma\|u_0\|_{L^\alpha(E)}^{q+1-p}$;
			\item[(ii)] if $p_\alpha\leq p$, then
			$T_*\leq\gamma\|u_0\|_{L^{q+1}(E)}^{q+1-p} |E|^{\frac{\lambda_{q+1}}{N(q+1)}}$.
		\end{itemize}
	\end{theorem}

	The following result, as a consequence of Theorems \ref{thm-1-6} and \ref{thm-1-7}, provides quantitative estimates for the decay of weak solutions to \eqref{1.1} as $t \rightarrow T_*$.
	\begin{theorem}
		\label{thm-1-9}
		Let $0<p-1<q$, and let $u$ be a local weak solution of \eqref{1.1} satisfying $u\geq 0$ in $\mathbb{R}^N\times(0,T)$. Let $T_*\in(0,T)$ be an
		extinction time for $u$, that is, $u(\cdot,t)=0$ a.e. in $\mathbb{R}^N$ for all  $t\in[T_*,T)$.
		Assume that \(K_{4\rho}(x_0)\Subset E\). Then there exists a constant
		\(\gamma>0\) depending only on \(N,p,s,q,\Lambda\) such that the following statements hold,
		
		\begin{itemize}
			\item[(i)] for every $0<t<T_*$,
			\begin{align*}
			\int_{K_{\rho}(x_0)} u^q(x, t) \, dx+\rho^N\mathcal T_+(u;x_0,\rho/2;t,T_*)
			\leq \gamma \left[ \frac{(T_* - t)^q}{\rho^{\lambda_q}} \right]^{\frac{1}{q+1-p}};
			\end{align*}
			\item[(ii)] if \( \lambda_q> 0 \), then for every $0<t<T_*$,
			\begin{align*}
			\sup_{K_{\rho/2}(x_0)}
			u\!\left(\cdot,\frac{t+T_*}{2}\right)
			\leq\gamma\left(\frac{T_*-t}{\rho^{sp}}\right)^{\frac1{q+1-p}}.
			\end{align*}
		\end{itemize}
	\end{theorem}
	
	Part (i) provides decay estimates for both the local mass and the tail of the positive part of $u$ throughout the fast diffusion range, while part (ii) gives a local supremum decay estimate under the additional condition $\lambda_q>0$. Neither result requires any additional assumption on the decay rate of the exterior values.

This paper is organized as follows. Section \ref{sec2} introduces the basic notation and presents several preliminary results. In Section \ref{sec3}, we first give a Caccioppoli-type inequality and then establish the local boundedness result stated in Theorem \ref{thm-1-3}. The integral Harnack-type inequalities stated in Theorems \ref{thm-1-6} and \ref{thm-1-7} are proved in Section \ref{sec4}. Finally, in Section \ref{sec5}, we study the finite extinction time and derive quantitative decay estimates.
	
	\section{Preliminaries}
	\label{sec2}
	
	\subsection{Notation}
	First, we collect some notation used throughout the paper. We denote by $K_\rho(x_0)$ the cube centered at $x_0\in\mathbb{R}^N$, with side length $\rho>0$, and whose faces are parallel to the coordinate planes in $\mathbb{R}^N$. As is customary, we write the general backward parabolic cylinders as
	\begin{align*}
	Q_{R,S}(x_0,t_0):=K_R(x_0) \times\left(t_0-S,t_0\right].
	\end{align*}
	We will omit $(x_0,t_0)$ if the context is clear or $(x_0,t_0)=(0,0)$.
	
	For fixed $k\in \mathbb{R}$, define
	\begin{align*}
		(u-k)_+=\max\{u-k,0\} \quad \textmd{and} \quad  (u-k)_-=\max\{-(u-k),0\}.
	\end{align*} 
	
	Throughout the paper, $\sup$ and $\inf$ denote the essential
	supremum and infimum, respectively. 
	
	The generic positive constants are denoted by $\gamma$. Their values may change from line to line, and even within the same line. Whenever necessary, we specify the dependence of $\gamma$ on the relevant parameters. We also write $$\textit{data}:=N,p,s,q,\Lambda.$$
	
	We use the shorthand notations
	\begin{align*}
		d\mu=d\mu(x,y,\tau)=K(x,y,\tau)\,dxdy 
	\end{align*}
	and
	\begin{align*}
		U(x,y,\tau):=|u(x,\tau)-u(y,\tau)|^{p-2}(u(x,\tau)-u(y,\tau)).
	\end{align*}
	\subsection{Technical lemmas}
	
	\begin{lemma}[Lemma 2.1, \cite{DZZ21}]
		\label{lem-2-1}
      Let $ s, \theta \in (0, 1)$ and $ 1 \leq p, p_2 < p_1 \leq \infty$. Assume that
    $$s > \frac{N}{p} - \frac{N}{p_1}$$
     and
     $$\theta \left( s - \frac{N}{p} + \frac{N}{p_1} \right) + (1 - \theta) \left( \frac{N}{p_1} - \frac{N}{p_2} \right) = 0.$$
     Then there exists a constant $\gamma > 0$ depending only on $ s, p, p_1, p_2$ and $N$ such that
     \begin{align*}
     \|f\|_{L^{p_1}(B_1)} \leq \gamma \|f\|_{W^{s, p}(B_1)}^{\theta} \|f\|_{L^{p_2}(B_1)}^{1-\theta}
     \end{align*}
     for every \( f \in W^{s, p}(B_1) \cap L^{p_2}(B_1) \).
		\end{lemma}
	
	In the following, we prove a parabolic Sobolev-type embedding lemma.
	\begin{lemma}
		\label{lem-2-2}
		Let $0<t_1<t_2$, $p,h>1$, $s\in(0,1)$ and $B_r\subset\mathbb{R}^N$.
		Then for every
		$$u\in L^{\infty}(t_1, t_2; L^h(B_r)) \cap L^p(t_1, t_2; W^{s,p}(B_r)),$$
	we have
	\begin{align*}
		&\int_{t_1}^{t_2}\int_{B_r}|u(x,\tau)|^{p\left(1+\frac{hs}{N}\right)} \,dxd\tau \\
		\leq& \gamma\left(r^{sp} \int_{t_1}^{t_2} \int_{B_r} \int_{B_r} \frac{|u(x,\tau)-u(y,\tau)|^p}{|x - y|^{N+sp}}\, dxdyd\tau+ \int_{t_1}^{t_2} \int_{B_r} |u(x, \tau)|^p \,dxd\tau\right)\\
		&\times\left(\operatorname*{ess \sup}_{t_1<\tau<t_2}  \mint_{B_r}|u(x,\tau)|^h\,dx\right)^{\frac{sp}{N}},
	\end{align*}
	where $\gamma>0$ only depends on $N,p,s$ and $h$. 
	\end{lemma}
	\begin{proof}
	We prove the embedding inequality for $r=1$. The general case then follows by a standard scaling argument.
	
	Assume first that $sp<N$. By H\"{o}lder's inequality, we have
	\begin{align*}
	&\int_{t_1}^{t_2} \mint_{B_1} |u(x, \tau)|^{p\left(1+\frac{hs}{N}\right)}\,dxd\tau\\
	=& \int_{t_1}^{t_2} \mint_{B_1} |u(x,\tau)|^{\frac{hsp}{N}} |u(x, \tau)|^p\,dxd\tau\\
	\leq& \int_{t_1}^{t_2} \left( \mint_{B_1} |u(x,\tau)|^h dx \right)^{\frac{sp}{N}} \left( \mint_{B_1} |u(x, \tau)|^{\frac{Np}{N-sp}}\,dx \right)^{\frac{N-sp}{N}}\, d\tau\\
	\leq& \left( \operatorname*{ess\,sup}_{t_1 < \tau < t_2} \mint_{B_1} |u(x,\tau)|^h \,dx \right)^{\frac{sp}{N}} \int_{t_1}^{t_2} \left( \mint_{B_1} |u(x, \tau)|^{\frac{Np}{N-sp}}\,dx \right)^{\frac{N-sp}{N}}\, d\tau.
	\end{align*}	
	Combining the above estimate with  \cite[Theorem 6.7]{DPV12}, we arrive at the desired inequality.
		
	Next, suppose that $sp\geq N$. Set $\theta = \frac{N}{N+hs}$. Then
	\begin{align*}
		s > \frac{N}{p} - \frac{N}{p(1 + \frac{hs}{N})}, \quad p\left(1 + \frac{hs}{N}\right) > h
\end{align*}
and
\begin{align*}
	\theta \left( s - \frac{N}{p} + \frac{N}{p(1 + \frac{hs}{N})} \right) + (1 - \theta) \left( \frac{N}{p(1 + \frac{hs}{N})} - \frac{N}{h} \right) = 0.
	\end{align*}
Consequently, we may apply Lemma \ref{lem-2-1} to obtain
	\begin{align}
		\label{2.1}
		\mint_{B_1} |u(x,\tau)|^{p(1+\frac{hs}{N})}\,dx\leq &\gamma \left( \int_{B_1} \mint_{B_1} \frac{|u(x,\tau)-u(y,\tau)|^p}{|x-y|^{N+sp}} \,dxdy + \mint_{B_1} |u(x,\tau)|^p\, dx \right) \nonumber\\
		&\times \left( \mint_{B_1} |u(x,\tau)|^h\,dx \right)^{\frac{sp}{N}}
		\end{align}
		for a.e. $\tau\in(t_1,t_2)$. Integrating \eqref{2.1} with respect to the time variable over $(t_1,t_2)$ completes the proof.
	\end{proof}
 We next recall an elementary algebraic inequality, see Lemma 2.2 in \cite{AF89} for $0<\alpha<1$ and inequality (2.4) in \cite{GM86} for $\alpha>1$.

	\begin{lemma}
		\label{lem-2-3}
		For every $\alpha>0$, there is a constant $\gamma$ depending only on $\alpha$ such that
		\begin{align*}
			\frac{1}{\gamma}\left|| b|^{\alpha-1} b-|a|^{\alpha-1} a\right|\leq(|a|+|b|)^{\alpha-1}
			|b-a|\leq \gamma\left||b|^{\alpha-1} b-|a|^{\alpha-1} a \right|
	\end{align*}
		for all $a,b\in\mathbb{R}$.
	\end{lemma}
	To work with the term involving the time derivative, we shall use auxiliary functions $\bm{h}_\pm$ defined as
	\begin{align}
		\label{2.2}
		\bm{h}_\pm(w,k):=\pm q \int_k^w|s|^{q-1}(s-k)_{\pm}\,ds,
	\end{align}
	for $k,w\in\mathbb{R}$ and $q>0$. It follows directly from the definition that $\bm{h}_\pm(w,k)\geq 0$. We also write
	\begin{align*}
		\bm{h}(w,k):= q \int_k^w|s|^{q-1}(s-k)\,ds.
	\end{align*}
	
	The next lemma can be deduced with the help of Lemma \ref{lem-2-3}.
	
	\begin{lemma} [Lemma 2.2, \cite{BDL21}]
		\label{lem-2-4}
		Let $q>0$. There exists a constant $\gamma=\gamma(q)>0$ such that, for all $a,b\in\mathbb{R}$, 
		\begin{align*}
			\frac{1}{\gamma}(|a|+|b|)^{q-1}|a-b|^2 \leq \bm{h}(a,b) \leq \gamma(|a|+|b|)^{q-1}|a-b|^2
		\end{align*}
		and
		\begin{align*}
			\frac{1}{\gamma}(|a|+|b|)^{q-1}(a-b)_{ \pm}^2 \leq \bm{h}_{ \pm}(a, b) \leq \gamma(|a|+|b|)^{q-1}(a-b)_{ \pm}^2.
		\end{align*}
	\end{lemma}
	
	We give a fast geometric convergence lemma from \cite[Lemma 4.1]{D93}.
	
	\begin{lemma} 
		\label{lem-2-5}
		Let $\{Y_j\}_{j=0}^\infty$ be a sequence of positive numbers satisfying
		\begin{align*}
			Y_{j+1}\leq Kb^jY_j^{1+\delta},\quad j=0,1,2, \ldots
		\end{align*}
		for some constants $K>0$, $b>1$ and $\delta>0$. If
		\begin{align*}
			Y_0\leq K^{-\frac{1}{\delta}}b^{-\frac{1}{\delta^2}},
		\end{align*}
		then we have $Y_j\rightarrow 0$ as $j\rightarrow \infty$.
	\end{lemma}

	The following iteration lemma is displayed in \cite[Lemma 6.1]{G03}.
	
	\begin{lemma}
		\label{lem-2-6}
		Let $A, B, C\geq 0$, $\alpha>\beta\geq0$ and $\theta\in (0, 1)$. Suppose that $f: [r, \rho] \rightarrow [0, \infty)$ is a bounded function that satisfies 
		\begin{align*}
		f(R_1) \leq \theta f(R_2)+\frac{A}{(R_2-R_1)^\alpha}+\frac{B}{(R_2-R_1)^\beta}+C 
		\end{align*}
		for all $r\leq R_1<R_2\leq\rho$. Then we have
		\begin{align*}
		f(r)\leq \gamma(\alpha, \theta)\left[\frac{A}{(\rho-r)^\alpha}+\frac{B}{(\rho-r)^\beta}+C\right].
		\end{align*}
	\end{lemma} 
The following elementary inequalities are key ingredients in handling the nonlocal terms in the proof of the integral Harnack inequality.
\begin{lemma}[Lemma 2.7, \cite{AABP18}]
	\label{lem-2-7}
	Let $p> 1$, $\alpha>0$ and $a,b\geq 0$. There exists a constant $\gamma(p,\alpha)>0$ such that
	\begin{align*}
		\left| a^{\frac{p+\alpha-1}{p}} - b^{\frac{p+\alpha-1}{p}} \right|^p \leq \gamma |a - b|^{p-2} (a - b)(a^\alpha - b^\alpha).
		\end{align*}
In addition, if $\alpha\geq 1$, then there exists a constant $\gamma(p,\alpha)>0$ such that
\begin{align*}
	(a + b)^{\alpha-1} |a - b|^p \leq \gamma \left| a^{\frac{p+\alpha-1}{p}} - b^{\frac{p+\alpha-1}{p}} \right|^p.
	\end{align*}
\end{lemma}


	\subsection{Time mollification}
	
	Since weak solutions do not generally possess a sufficiently regular time derivative, we employ exponential mollification in time, following \cite{KL06}. For any $v\in L^1(E_T)$ and $h>0$, define
	\begin{align*}
		[\![v]\!]_h(x,t):=\frac{1}{h} \int_0^t e^{\frac{s-t}{h}} v(x,s)\,ds
	\end{align*}
	and
	\begin{align*}
		[\![v]\!]_{\bar{h}}(x, t):=\frac{1}{h} \int_t^T e^{\frac{t-s}{h}} v(x, s)\,ds.
	\end{align*}	
	
	The following fundamental properties of mollified functions are given in \cite [Appendix]{BDM13}.
	
	\begin{lemma}
		\label{lem-2-8}
		Let $r\geq 1$. Then the following conclusions hold. 
		
		\begin{itemize}
			\item [(i)] If $v \in L^r(E_T)$, then $[\![v]\!]_h \in L^r (E_T)$ and $\|[\![v]\!]_h\|_{L^r(E_T)}\leq\|v\|_{L^r(E_T)}$. In addition, we have
			$[\![v]\!]_h \rightarrow v$ strongly in $L^r(E_T)$ and almost everywhere in $E_T$ as $h \rightarrow 0$. This statement also holds for $[\![v]\!]_{\bar h}$.
			
			\item [(ii)] The following identities hold almost everywhere in $E_T$,
			\begin{align*}
				\partial_t[\![v]\!]_h=\frac{1}{h}\left(v-[\![v]\!]_h\right),\quad \partial_t[\![v]\!]_{\bar{h}}=\frac{1}{h}\left([\![v]\!]_{\bar{h}}-v\right).
			\end{align*}
			
			\item [(iii)] If $v\in L^r(E_T)$, then $[\![v]\!]_h$ and $[\![v]\!]_{\bar{h}}$ belong to $C\left([0, T] ; L^r(E)\right)$.
			
			
	\end{itemize}
	
	\end{lemma}


	\section{Energy estimate and local boundedness}
	\label{sec3}
	\subsection{Energy estimate} We first present a Caccioppoli inequality with a level that depends on time, as established in \cite[Lemma 2.4]{SSY26}.
	
	\begin{lemma}
		\label{lem-energy-moving-level}
		Let $p>1$, $q>0$, and $s\in(0,1)$.
		Let $u$ be a weak subsolution (super-) to \eqref{1.1} such that
		$u\geq0$ in $Q_{\rho,\theta}:=K_\rho(x_0)\times(t_0-\theta,t_0]\Subset E_T.$ Let $k:[t_0-\theta,t_0]\to(0,\infty)$ be an absolutely continuous function.
		Let $\psi$ be a Lipschitz cutoff function satisfying
		$0\leq\psi\leq1$, with spatial support contained in a fixed
		compact subset of $K_\rho(x_0)$ for any $\tau\in(t_0-\theta,t_0]$. Set
		$w_{\pm}(x,\tau):=(u(x,\tau)-k(\tau))_{\pm}$. Then there exists a constant $\gamma>0$ depending only on
		$N,p,s,q,\Lambda$ such that
		\begin{align}
			\label{3.1}
			&\int_{K_\rho(x_0)}\bm h_\pm(u(x,t),k(t))\psi^p(x,t)\,dx\nonumber\\
			&+\frac1\gamma\int_{t_0-\theta}^{t}\int_{K_\rho(x_0)}\int_{K_\rho(x_0)}
			|w_\pm(x,\tau)-w_\pm(y,\tau)|^p\min\{\psi(x,\tau),\psi(y,\tau)\}^p\,d\mu d\tau\nonumber\\
			\leq&\int_{K_\rho(x_0)}\bm h_\pm(u(x,t_0-\theta),k(t_0-\theta))\psi^p(x,t_0-\theta)\,dx
			\nonumber\\
			&\mp\int_{t_0-\theta}^{t}k'(\tau)\int_{K_\rho(x_0)}(u^q(x,\tau)-k^q(\tau))_\pm
			\psi^p(x,\tau)\,dxd\tau\nonumber\\
			&+\int_{t_0-\theta}^{t}\int_{K_\rho(x_0)}\bm h_\pm(u(x,\tau),k(\tau))
			\left|\partial_\tau\psi^p(x,\tau)\right|\,dxd\tau\nonumber\\
			&+\gamma\int_{t_0-\theta}^{t}\int_{K_\rho(x_0)}\int_{K_\rho(x_0)}
			\max\{w_\pm(x,\tau),w_\pm(y,\tau)\}^p|\psi(x,\tau)-\psi(y,\tau)|^p\,d\mu d\tau\nonumber\\
			&+\gamma
			\int_{t_0-\theta}^{t}\int_{K_\rho(x_0)}w_\pm(x,\tau)\psi^p(x,\tau)
			\int_{\mathbb R^N\setminus K_\rho(x_0)}\frac{w_\pm^{p-1}(y,\tau)}{|x-y|^{N+sp}}\,dydxd\tau,
		\end{align}
		where the function $\bm{h}_{\pm}$ is defined as in \eqref{2.2}.
	\end{lemma}

    \subsection{Local boundedness}
    In the following, we devote to showing the quantitative $L^\infty$ estimate stated in Theorem~\ref{thm-1-3}. We first establish qualitative local boundedness under
    the condition $\lambda_{q+1}=N(p-q-1)+(q+1)sp>0$. This preliminary result provides the boundedness needed for the subsequent iteration in this parameter range. We then proceed to derive the quantitative estimate.
	\begin{lemma}\label{lb-qualitative}
	Let $0<p-1<q$ and
	$\lambda_{q+1}=N(p-q-1)+(q+1)sp>0$. Suppose that $u$ is a weak subsolution to \eqref{1.1}
	satisfying $u\geq0$ in $Q_{\rho,\theta}=K_\rho(x_0)\times(t_0-\theta,t_0]\Subset E_T$.
	Then we have $u\in L^\infty(Q_{\rho/2,\theta/2})$.
	\end{lemma}
	\begin{proof}
		By translation, we may assume that $(x_0,t_0)=(0,0)$. For $j=0,1,2,\ldots$, define
		\begin{align*}
		\rho_j=\frac\rho2+\frac\rho{2^{j+1}},
		\qquad
		\widetilde\rho_j=\frac{\rho_j+\rho_{j+1}}2,
		\qquad
		\widehat\rho_j=\frac{3\rho_j+\rho_{j+1}}4,
		\end{align*}
		\begin{align*}
		\theta_j=\frac\theta2+\frac\theta{2^{j+1}}, \qquad
		\widetilde\theta_j=\frac{\theta_j+\theta_{j+1}}2,
		\qquad
		\widehat\theta_j=\frac{3\theta_j+\theta_{j+1}}4.
		\end{align*}
		Set the domains
		\begin{align*}
		K_j=K_{\rho_j},\qquad
		\widetilde K_j=K_{\widetilde\rho_j},\qquad
		\widehat K_j=K_{\widehat\rho_j},
		\end{align*}
		\begin{align*}
		Q_j=K_j\times(-\theta_j,0],
		\qquad
		\widetilde Q_j=\widetilde K_j\times(-\widetilde\theta_j,0],
		\qquad
		\widehat Q_j=\widehat K_j\times(-\widehat\theta_j,0].
		\end{align*}
Choose a smooth cutoff function $\zeta$ on $K_j\times[-\theta_j,0]$, compactly supported in $\widehat K_j$ with respect to the spatial variables and vanishing for $\tau\leq-\widehat\theta_j$, such that
		\begin{align*}
		0\leq\zeta\leq1,\qquad\zeta=1\ \text{in }\widetilde Q_j,\qquad|D\zeta|\leq\frac{\gamma2^j}{\rho},\qquad|\partial_\tau\zeta|\leq\frac{\gamma2^j}{\theta}.
		\end{align*}
		For a sufficiently large constant $C>0$, define 
		\begin{align*}
		\ell(\tau)=\begin{cases}
		Ck^{1-q}\mathcal T_+(u;0,\rho;-\theta,\tau),&q\geq1,\\[2pt]
		C\mathcal T_+(u;0,\rho;-\theta,\tau),&0<q<1,
		\end{cases}
		\end{align*}
		where $k>0$ will be chosen later. The levels are then defined by
			\begin{align*}
			k_j(\tau)=\begin{cases}
			k(1-2^{-j})+\ell(\tau),&q\geq1,\\[2pt]
			\bigl[k^q(1-2^{-j})^q+\ell(\tau)\bigr]^{\frac1q},&0<q<1,
			\end{cases}
		\end{align*}
		and
		\begin{align*}
		\widetilde k_j(\tau)=\frac{k_j(\tau)+k_{j+1}(\tau)}{2}.
		\end{align*}
		By construction, the levels are absolutely continuous and nondecreasing. 
		 Set 
		 \begin{align*} 
		 	w_j=(u-\widetilde k_j)_+,\,\text{ for } q\geq 1 \,\text{ and } \,\, w_j=\bigl(u_+^q-\widetilde k_j(\tau)^q\bigr)_+^{\frac1q}\, \text{ for } 0<q<1.
		 	\end{align*}
		 
		For $q\geq1$, applying Lemma~\ref{lem-energy-moving-level} with the level $\widetilde k_j$ and the cutoff function $\zeta$, and then using Lemma~\ref{lem-2-4}, we obtain
		\begin{align}
			\label{3.2}
			&\operatorname*{ess\,sup}_{-\theta_j<\tau<0}\int_{K_j}w_j^{q+1}\zeta^p\,dx+\int_{-\theta_j}^0\int_{K_j}\ell'(\tau)(u^q(x,\tau)-\widetilde k_j^q(\tau))_+\zeta^p\,dxd\tau\nonumber\\
			&+\int_{-\theta_j}^0\int_{K_j}\int_{K_j}\frac{|w_j(x,\tau)-w_j(y,\tau)|^p\min\{\zeta(x,\tau),\zeta(y,\tau)\}^p}{|x-y|^{N+sp}}\,dxdyd\tau\nonumber\\
			\leq&\gamma\iint_{Q_j}u^{q+1}\chi_{\{u>\widetilde k_j\}}|\partial_\tau\zeta^p|\,dxd\tau
			\nonumber\\
			&+\gamma\int_{-\theta_j}^0\int_{K_j}\int_{K_j}\frac{\max\{u(x,\tau),u(y,\tau)\}^p
				|\zeta(x,\tau)-\zeta(y,\tau)|^p}{|x-y|^{N+sp}}\nonumber\\
				&\qquad\qquad\qquad\qquad\qquad\qquad\qquad\times\chi_{\{\max\{u(x,\tau),u(y,\tau)\}>\widetilde k_j(\tau)\}}\,dxdyd\tau	\nonumber\\
			&+\gamma\int_{-\theta_j}^0\int_{K_j}w_j(x,\tau)\zeta^p(x,\tau)\int_{\mathbb R^N\setminus K_j}\frac{u_+(y,\tau)^{p-1}}{|x-y|^{N+sp}}\,dydxd\tau\nonumber\\
			&=:I_1+I_2+I_3,
		\end{align}
		where $\gamma>0$ depends only on the \textit{data}. 
		
	   For $0<q<1$, a standard regularization argument justifies using $w_j\zeta^p$ as a test function
	   in the weak formulation \eqref{1.5}. The time term is then evaluated using the identity
		\begin{align*}
		\partial_\tau(u^q)w_j=\frac{q}{q+1}\partial_\tau(w_j^{q+1})+\bigl(\widetilde k_j(\tau)^q\bigr)'w_j.
		\end{align*}
		For each fixed $\tau$, the map
		\begin{align*}
		z\longmapsto\bigl(z^q-\widetilde k_j(\tau)^q\bigr)_+^{1/q}, \qquad z\geq0,
		\end{align*}
		is nondecreasing, convex and $1$-Lipschitz, and vanishes at zero.
		These properties together with Young's inequality yield, for $x,y\in K_j$,
		\begin{align*}
			&|u(x,\tau)-u(y,\tau)|^{p-2}
			\bigl(u(x,\tau)-u(y,\tau)\bigr)\bigl(w_j(x,\tau)\zeta(x,\tau)^p-w_j(y,\tau)\zeta(y,\tau)^p\bigr)\\
	    	\geq&\frac12|w_j(x,\tau)-w_j(y,\tau)|^p\min\{\zeta(x,\tau),\zeta(y,\tau)\}^p\\
			&-\gamma(p)\max\{u(x,\tau),u(y,\tau)\}^p|\zeta(x,\tau)-\zeta(y,\tau)|^p\chi_{\{\max\{u(x,\tau),u(y,\tau)\}>\widetilde k_j(\tau)\}}.
		\end{align*}
		For the exterior interaction, the nonnegativity of $u$
		on the support of the cutoff function gives
		\begin{align*}
			|u(x,\tau)-u(y,\tau)|^{p-2}
			\bigl(u(x,\tau)-u(y,\tau)\bigr)w_j(x,\tau)\geq-u_+(y,\tau)^{p-1}w_j(x,\tau).
		\end{align*}
		Combining these estimates, we obtain \eqref{3.2}, with the second term on the left-hand side
		replaced by
        \begin{align*}
		\int_{-\theta_j}^0\int_{K_j}
		(\widetilde k_j^q(\tau))'w_j\zeta^p\,dxd\tau.
		\end{align*}
		
		We now estimate the left-hand side of
	\eqref{3.2}. Choose $k$ such that
	\begin{align}
		\label{3.3}
		C\mathcal T_+(u;0,\rho;-\theta,0)\leq\frac{k^q}{4}.
	\end{align}
	By construction, we can see
	\begin{align}
	\label{3.4}
	\widetilde k_j\geq\frac k4,\qquad k_{j+1}\leq \gamma(q)k,\qquad k_{j+1}-k_j\geq\frac{k}{2^{j+1}}.
	\end{align}
	For $q\geq1$, we have
	\begin{align*}
	(u-k_{j+1})_+^{q+1}=\bigl[w_j-(k_{j+1}-\widetilde k_j)\bigr]_+^{q+1}\leq w_j^{q+1}.
	\end{align*}
	For $0<q<1$, it follows from \eqref{3.4} that
\begin{align*}
	\frac{d}{dr}(r^q-\widetilde k_j^q)^{1/q}\geq \gamma^{-1}2^{-j(1-q)/q},\qquad r\geq k_{j+1}.
\end{align*}
Consequently,
\begin{align*}
	(u-k_{j+1})_+
	\leq \gamma 2^{j(1-q)/q}w_j,
\end{align*}
and hence
\begin{align*}
	(u-k_{j+1})_+^{q+1}\leq \gamma 2^{j(1-q)(q+1)/q}w_j^{q+1}.
\end{align*}
Moreover, to estimate the nonlocal term, we use
\begin{align*}
	|(u(x,\tau)-k_{j+1})_+-(u(y,\tau)-k_{j+1})_+|\leq \gamma 2^{j(1-q)/q}|w_j(x,\tau)-w_j(y,\tau)|
\end{align*}
for almost every $x,y\in K_j$.

	We next estimate the terms $I_1$–$I_3$ on the right-hand side of \eqref{3.2}. Based on the properties of the cutoff function $\zeta$ and the levels, we obtain
	\begin{align*}
		I_1
		&\leq \frac{\gamma b^j}{\theta}
		\iint_{Q_j}
		u^{q+1}\chi_{\{u>\widetilde k_j\}}
		\,dxd\tau\\
		&\leq \frac{\gamma b^j}{\theta}
		\iint_{Q_j}(u-k_j)_+^{q+1}
		\,dxd\tau.
	\end{align*}
	By using the properties of the cutoff function $\zeta$ again, we obtain
	\begin{align*}
		I_2
		&\leq
		\frac{\gamma b^j}{\rho^{sp}}
		\iint_{Q_j}
		u^p\chi_{\{u>\widetilde k_j\}}
		\,dxd\tau\\
		&\leq
		\frac{\gamma b^j}{\rho^{sp}k^{q+1-p}}
		\iint_{Q_j}(u-k_j)_+^{q+1}
		\,dxd\tau\\
		&\leq
		\frac{\gamma b^j}{\theta}
		\iint_{Q_j}(u-k_j)_+^{q+1}
		\,dxd\tau,
	\end{align*}
	where the second inequality follows from
	\begin{align*}
		u^p\chi_{\{u>\widetilde k_j\}}
		\leq \gamma b^j k^{p-q-1}
		(u-k_j)_+^{q+1},
	\end{align*}
	and the last inequality holds provided that
	\begin{align}
		\label{3.5}
		k\geq
		\left(\frac{\theta}{\rho^{sp}}\right)^
		{\frac{1}{q+1-p}}.
	\end{align}
To estimate $I_3$, we split $ I_3=I_{3,1}+I_{3,2}$, where $I_{3,1}$ and $I_{3,2}$ denote
the contributions from
$K_\rho\setminus K_j$ and
$\mathbb R^N\setminus K_\rho$, respectively. By the properties of $\zeta$ and the levels,
we have
\begin{align*}
	I_{3,1}
	&\leq
	\frac{\gamma b^j}{\rho^{N+sp}}
	\left(
	\operatorname*{ess\,sup}_{-\theta<\tau<0}
	\int_{K_\rho}u^{p-1}\,dx
	\right)
	\iint_{Q_j}w_j\zeta^p\,dxd\tau\\
	&\leq
	\frac{\gamma b^j}{\rho^{N+sp}k^q}
	\left(
	\operatorname*{ess\,sup}_{-\theta<\tau<0}
	\int_{K_\rho}u^{p-1}\,dx
	\right)
	\iint_{Q_j}(u-k_j)_+^{q+1}
	\,dxd\tau\\
	&\leq
	\frac{\gamma b^j}{\theta}
	\iint_{Q_j}(u-k_j)_+^{q+1}
	\,dxd\tau,
\end{align*}
provided that
\begin{align}
	\label{k}
	\frac{\theta}{\rho^{N+sp}k^q}
	\operatorname*{ess\,sup}_{-\theta<\tau<0}
	\int_{K_\rho}u^{p-1}\,dx
	\leq1.
\end{align}
The supremum is finite by the local
$L^{q+1}$ time continuity of $u$
and H\"older's inequality.
For $I_{3,2}$, we derive
\begin{align*}
	I_{3,2}
	&\leq
	\gamma\int_{-\theta_j}^0
	\int_{K_j}w_j\zeta^p
	\int_{\mathbb R^N\setminus K_\rho}
	\frac{u_+(y,\tau)^{p-1}}
	{|y|^{N+sp}}\,dydxd\tau,
\end{align*}
which can be absorbed in the second term on the left-hand side by choosing $C>0$ sufficiently large and using
\begin{align*}
	(u^q-\widetilde k_j^q)_+\geq \gamma(q) k^{q-1}w_j,\qquad\text{ for} \,\, q\geq 1,
\end{align*}
and 
\begin{align*}
	(\widetilde k_j^q(\tau))'=\frac{\ell'}2\widetilde k_j^{q-1}\bigl(k_j^{1-q}+k_{j+1}^{1-q}\bigr)
	\geq\frac{\ell'}2, \qquad\text{ for} \,\, 0<q<1.
\end{align*}
	Once the constant $C$ is fixed, we can choose $k$ large enough to satisfy the requirements \eqref{3.3}, \eqref{3.5} and \eqref{k}.
	
	Combining the above estimates, we deduce that
	\begin{align}
		\label{3.6}
		&\operatorname*{ess\,sup}_{-\widetilde\theta_j<\tau<0}\int_{\widetilde K_j}(u-k_{j+1})_+^{q+1}\,dx\nonumber\\
		&+
		\int_{-\widetilde\theta_j}^0\int_{\widetilde K_j}\int_{\widetilde K_j}
		\frac{|(u-k_{j+1})_+(x,\tau)-(u-k_{j+1})_+(y,\tau)|^p}{|x-y|^{N+sp}}\,dxdyd\tau
		\nonumber\\
		\leq&\frac{\gamma b^j}{\theta}\iint_{Q_j}(u-k_j)_+^{q+1}\,dxd\tau,
	\end{align}
	where $\gamma>0$ and $b>1$ depend only on the \textit{data}.
	
	Set
	\begin{align*}
	m=p\frac{N+s(q+1)}N,\qquad Y_j=\mint\!\!\mint_{Q_j}(u-k_j)_+^{q+1}\,dxd\tau.
	\end{align*}
	Since $\lambda_{q+1}>0$, we have $m>q+1$. Applying Lemma~\ref{lem-2-2} with $h=q+1$ to
	$(u-k_{j+1})_+$ in $\widetilde Q_j$, and using
	\eqref{3.6}, we obtain
	\begin{align*}
		&\iint_{Q_{j+1}}(u-k_{j+1})_+^m\,dxd\tau\leq\iint_{\widetilde Q_j}(u-k_{j+1})_+^m\,dxd\tau\\
		\leq&
		\gamma\Bigg[\widetilde \rho_j^{sp}\int_{-\widetilde\theta_j}^0
		\iint_{\widetilde K_j\times\widetilde K_j}\frac{|(u-k_{j+1})_+(x,\tau)-(u-k_{j+1})_+(y,\tau)|^p}
		{|x-y|^{N+sp}}\,dxdyd\tau\\
		&\quad
		+\iint_{\widetilde Q_j}(u-k_{j+1})_+^p\,dxd\tau\Bigg]\times
		\left[\operatorname*{ess\,sup}_{-\widetilde\theta_j<\tau<0}\mint_{\widetilde K_j}(u-k_{j+1})_+^{q+1}\,dx\right]^{\frac{sp}{N}}\\
		\leq&\gamma b^j\left[\frac1{\theta^{\frac{N+sp}{N}}}+\frac{k^{p-q-1}}{\rho^{sp}\theta^{\frac{sp}{N}}}
		\right]\left[\iint_{Q_j}(u-k_j)_+^{q+1}\,dxd\tau\right]^{\frac{N+sp}{N}}.
	\end{align*}
Here we utilized the fact that $\rho/2\leq\widetilde\rho_j\leq\rho$. By our choice of $k$ in \eqref{3.5} and H\"{o}lder's inequality, we have
\begin{align*}
	Y_{j+1}\leq
	\gamma b^j
	\left(\frac{\rho^{sp}}\theta\right)^{\frac{q+1}m}
	k^{-\frac{(q+1)\lambda_{q+1}}{Nm}}
	Y_j^{1+\frac{(q+1)sp}{Nm}}.
\end{align*}
	Note that the local $L^{q+1}$ time continuity of $u$ ensures
	\begin{align*}
	Y_0\leq\mint\!\!\mint_{Q_{0}}u^{q+1}\,dxd\tau<\infty.
	\end{align*}
	Choosing $k$ additionally so that
	\begin{align*}
	k\geq\gamma\left(\frac{\rho^{sp}}\theta\right)^{\frac N{\lambda_{q+1}}}
	\left(\mint\!\!\mint_{Q_{0}}u^{q+1}\,dxd\tau\right)^{\frac{sp}{\lambda_{q+1}}},
	\end{align*}
	we conclude from Lemma~\ref{lem-2-5}
	that $Y_j\to0$ as $j\to\infty$. 
	By \eqref{3.4}, the truncation levels satisfy $
	k_j(\tau)\rightarrow k_\infty(\tau)\leq \gamma k$ for all $\tau\in[-\theta,0]$.
	Hence $u\leq \gamma k$ almost everywhere in $Q_{\rho/2,\theta/2}$, which completes the proof.
	\end{proof}

	\subsubsection{An iterative inequality for $r\geq q+1$}
	Let $(x_0,t_0)=(0,0)$.  When $r\leq m$, the condition $\lambda_r>0$ implies $\lambda_{q+1}>0$. Lemma~\ref{lb-qualitative} and a finite covering argument therefore give
	\begin{align*}
	\operatorname*{ess\,sup}_{Q_{\sigma_2\rho,\sigma_2\theta}}u<\infty,
	\qquad \frac12<\sigma_2\leq\frac34.
	\end{align*}
	When $r>m$, the boundedness follows from the hypothesis
	of Theorem \ref{thm-1-3}. We now derive the estimate in a fixed cylinder under the assumption
	\begin{align*}
	M_0:=\operatorname*{ess\,sup}_{Q_{\rho,\theta}}u<\infty.
	\end{align*}
	The case $M_0=0$ is trivial. This estimate will only be applied on the smaller cylinders above, where boundedness has already been established. The final estimate is independent of $M_0$.

	For $\sigma\in(0,1)$, define decreasing sequences
	\begin{align*}
		\rho_j&=\sigma\rho+2^{-j}(1-\sigma)\rho,
		&\widetilde\rho_j&=\frac{\rho_j+\rho_{j+1}}2,
		&\widehat\rho_j&=\frac{3\rho_j+\rho_{j+1}}4,\\
		\theta_j&=\sigma\theta+2^{-j}(1-\sigma)\theta,
		&\widetilde\theta_j&=\frac{\theta_j+\theta_{j+1}}2,
		&\widehat\theta_j&=\frac{3\theta_j+\theta_{j+1}}4,
	\end{align*}
	for $j=0,1,\ldots$.
	Set
	\begin{align*}
	K_j=K_{\rho_j},
	\qquad\widetilde K_j=K_{\widetilde\rho_j},
	\qquad\widehat K_j=K_{\widehat\rho_j},
	\end{align*}
	\begin{align*}
	Q_j=K_j\times(-\theta_j,0],
	\qquad
	\widetilde Q_j=\widetilde K_j\times(-\widetilde\theta_j,0],
	\qquad
	\widehat Q_j=\widehat K_j\times(-\widehat\theta_j,0].
	\end{align*}
	For $k>0$ to be chosen later, define
	\begin{equation}
		\label{3.7}
		\ell(\tau)=\frac{C}{(1-\sigma)^{N+sp}}
		\begin{cases}
			k^{1-q}\mathcal T_+(u;0,\rho;-\theta,\tau),&q\geq1,\\[2pt]
			\|u\|_{\infty,Q_{\rho,\theta}}^{1-q}\mathcal T_+(u;0,\rho;-\theta,\tau),&0<q<1.
		\end{cases}
	\end{equation}
	Here $C>0$ is a structural constant to be fixed later. Define the levels by
	\begin{align*}
	k_j(\tau)=k-\frac{k}{2^j}+\ell(\tau),
	\qquad
	\widetilde k_j(\tau)=\frac{k_j(\tau)+k_{j+1}(\tau)}2.
	\end{align*}
	We choose $k$ so that
	\begin{align}
	\label{3.8}
	0\leq\ell(\tau)\leq\frac k4,\qquad -\theta\leq \tau\leq0.
	\end{align}
	By the definition of the levels, we derive
	\begin{align*}
	\widetilde k_j-k_j
	=k_{j+1}-\widetilde k_j
	=\frac{k}{2^{j+2}}.
	\end{align*}
	Choose a smooth cutoff function $\zeta$ on $K_j\times[-\theta_j,0]$, compactly supported in $\widehat K_j$ with respect to the space variable and equal to zero for $\tau\leq-\widehat\theta_j$, satisfying
	\begin{align*}
	0\leq\zeta\leq1,
	\qquad
	|D\zeta|\leq\frac{\gamma2^{j+3}}{(1-\sigma)\rho},
	\qquad
	|\partial_\tau\zeta|\leq\frac{\gamma2^{j+3}}{(1-\sigma)\theta},
	\qquad
	\zeta=1\text{ in }\widetilde Q_j.
	\end{align*}
	Applying Lemma~\ref{lem-energy-moving-level} in this framework gives
	\begin{align}
		\label{3.9}
		&\operatorname*{ess\,sup}_{-\widetilde\theta_j<\tau<0}\int_{\widetilde K_j}\bm h_+(u,\widetilde k_j)\,dx+\int_{-\theta_j}^0\ell'(\tau)\int_{K_j}(u^q-\widetilde k_j^q)_+\zeta^p\,dxd\tau\nonumber\\
		&+\frac1\gamma\int_{-\widetilde\theta_j}^0\int_{\widetilde K_j}\int_{\widetilde K_j}
		\frac{|(u-\widetilde k_j)_+(x,\tau)-(u-\widetilde k_j)_+(y,\tau)|^p}
		{|x-y|^{N+sp}}\,dxdyd\tau\nonumber\\
		\leq&
		\frac{\gamma2^j}{(1-\sigma)\theta}\iint_{Q_j}\bm h_+(u,\widetilde k_j)\,dxd\tau+
		\frac{\gamma2^{pj}}{(1-\sigma)^p\rho^{sp}}\iint_{Q_j}(u-\widetilde k_j)_+^p\,dxd\tau
		\nonumber\\
		&+\gamma\int_{-\theta_j}^0\int_{K_j}(u-\widetilde k_j)_+(x,\tau)\zeta^p(x,\tau)
		\int_{\mathbb R^N\setminus K_j}\frac{(u-\widetilde k_j)_+(y,\tau)^{p-1}}{|x-y|^{N+sp}}
		\,dydxd\tau,
	\end{align}
	where the constant $\gamma$ depends only on the $\textit{data}$. Next, we estimate the first term on the left-hand side of \eqref{3.9}. On the set $\{u>k_{j+1}\}$, we have
	\begin{align*}
	1\leq\frac{u+\widetilde k_j}{u-\widetilde k_j}
	\leq\frac{2k_{j+1}}{k_{j+1}-\widetilde k_j}
	\leq2^{j+4},
	\end{align*}
	which together with Lemma \ref{lem-2-4} yields
	\begin{equation}
		\label{3.10}
		\bm h_+(u,\widetilde k_j)
		\geq\gamma^{-1}2^{-(j+4)(1-q)_+}
		(u-k_{j+1})_+^{q+1},
	\end{equation}
	where $\gamma>0$ depends only on $q$. Besides, on the set $\{u>\widetilde k_j\}$,
	\begin{align*}
	1\leq\frac{u+\widetilde k_j}{u-k_j}\leq\frac{2\widetilde{k}_j}{\widetilde{k}_j-k_j}\leq2^{j+4}.
	\end{align*}
Therefore, we have by Lemma \ref{lem-2-4} that
	\begin{equation}
		\label{3.11}
		\bm h_+(u,\widetilde k_j)
		\leq\gamma2^{(j+4)(q-1)_+}
		(u-k_j)_+^{q+1}\chi_{\{u>\widetilde k_j\}}.
	\end{equation}
	For the last term in \eqref{3.9}, we split the region. For almost every $\tau\in(-\theta_j,0)$ and
	$x\in\operatorname{supp}\zeta(\cdot,\tau)$, we have
	\begin{align}
		\label{3.12}
		&\int_{\mathbb R^N\setminus K_j}\frac{(u-\widetilde k_j)_+(y,\tau)^{p-1}}{|x-y|^{N+sp}}\,dy
		\nonumber\\
		\leq&
		\frac{\gamma2^{spj}}{(1-\sigma)^{sp}\rho^{sp}}\|u\|_{\infty,Q_0}^{p-1}+
		\frac{\gamma}{(1-\sigma)^{N+sp}}\int_{\mathbb R^N\setminus K_\rho}
		\frac{u_+^{p-1}(y,\tau)}{|y|^{N+sp}}\,dy.
	\end{align}	
	Indeed,
	\begin{align*}
	\rho_j-\widehat\rho_j=2^{-j-3}(1-\sigma)\rho,
	\qquad
	\widehat\rho_j\leq\widehat\rho_0
	=\rho-\frac{(1-\sigma)\rho}{8}.
	\end{align*}
	Moreover, when $q \geq 1$ we have
	$$ 
	(u^q-\widetilde k_j^q)_+\geq \gamma(q)k^{q-1}(u-\widetilde k_j)_+,
	$$
	and when $0<q<1$, we get
	$$(u^q-\widetilde k_j^q)_+\geq q\|u\|_{\infty,Q_0}^{q-1}(u-\widetilde k_j)_+.$$
	Therefore, by choosing $C$ large enough in \eqref{3.7}, the last term in \eqref{3.12} can be absorbed into the second term on the left-hand side of \eqref{3.9}. Combining \eqref{3.10}--\eqref{3.12}, we obtain
	\begin{align}
		\label{3.13}
		&2^{-j(1-q)_+}\operatorname*{ess\,sup}_{-\widetilde\theta_j<\tau<0}
		\int_{\widetilde K_j}(u-k_{j+1})_+^{q+1}\,dx\nonumber\\
		&\quad+\int_{-\widetilde\theta_j}^0	\int_{\widetilde K_j}\int_{\widetilde K_j}
		\frac{|(u-\widetilde k_j)_+(x,\tau)-(u-\widetilde k_j)_+(y,\tau)|^p}{|x-y|^{N+sp}}\,dxdyd\tau
		\nonumber\\
		&\leq
		\frac{\gamma2^{j[1+(q-1)_+]}}{(1-\sigma)\theta}\iint_{\widetilde A_j}(u-k_j)_+^{q+1}\,dxd\tau+
		\frac{\gamma2^{pj}}{(1-\sigma)^p\rho^{sp}}\iint_{Q_j}(u-\widetilde k_j)_+^p\,dxd\tau
		\nonumber\\
		&\quad+\frac{\gamma2^{spj}}{(1-\sigma)^{sp}\rho^{sp}}\|u\|_{\infty,Q_0}^{p-1}
		\iint_{Q_j}(u-\widetilde k_j)_+\,dxd\tau,
	\end{align}
	where
	\begin{align*}
	\widetilde A_j:=\{u>\widetilde k_j(\tau)\}\cap Q_j.
	\end{align*}
	Note that
	\begin{equation}
		\label{3.14}
		|\widetilde A_j|\leq\frac{2^{(j+2)r}}{k^r}\iint_{Q_j}(u-k_j)_+^r\,dxd\tau,
	\end{equation}
	which along with the assumption $r\geq q+1>p$ allows us to apply H\"{o}lder's inequality to get
	\begin{align}
		\label{3.15}
		\iint_{\widetilde A_j}(u-k_j)_+^{q+1}\,dxd\tau
		& \leq\left(\iint_{Q_j}(u-k_j)_+^r\,dxd\tau \right)^{\frac{q+1}{r}}|\widetilde{A}_j|^{1-\frac{q+1}{r}} \nonumber\\
		& \leq  \frac{\gamma 2^{(r-q-1)j}}{k^{r-q-1}} \iint_{Q_j}(u-k_j)_+^r\,dxd\tau.
		\end{align}
		Similarly, we can derive
		\begin{align}
		\label{3.16}
		\iint_{Q_j}(u-\widetilde k_j)_+^p\,dxd\tau\leq\frac{\gamma2^{(r-p)j}}{k^{r-p}}
		\iint_{Q_j}(u-k_j)_+^r\,dxd\tau,
		\end{align}
		and
		\begin{align}
		\label{3.17}
		\iint_{Q_j}(u-\widetilde k_j)_+\,dxd\tau
		&\leq\frac{\gamma2^{(r-1)j}}{k^{r-1}}\iint_{Q_j}(u-k_j)_+^r\,dxd\tau.
	\end{align}
	We now require $k$ to satisfy
	\begin{align}
		\label{choice of k1}
		k\geq\left(\frac{\theta}{\rho^{sp}}\right)^{\frac1{q+1-p}}+\left(\frac{\theta}{\rho^{sp}}
		\|u\|_{\infty,Q_0}^{p-1}\right)^{\frac1q},
		\end{align}
		and
		\begin{align}
		\label{choice of k2}
		k\geq
		\begin{cases}
			\left[\dfrac{4C}{(1-\sigma)^{N+sp}}
			\mathcal T_+(u;0,\rho;-\theta,0)\right]^{\frac1q},&q\geq1,\\[6pt]
			\dfrac{4C}{(1-\sigma)^{N+sp}}
			\|u\|_{\infty,Q_0}^{1-q}
			\mathcal T_+(u;0,\rho;-\theta,0),&0<q<1.
		\end{cases}
	\end{align}
	Then it follows from the estimates \eqref{3.13}--\eqref{3.17} that
	\begin{align}
		\label{3.19}
		&2^{-j(1-q)_+}\operatorname*{ess\,sup}_{-\widetilde\theta_j<\tau<0}
		\int_{\widetilde K_j}(u-k_{j+1})_+^{q+1}\,dx\nonumber\\
		&+\int_{-\widetilde\theta_j}^0\int_{\widetilde K_j}\int_{\widetilde K_j}
		\frac{|(u-\widetilde k_j)_+(x,\tau)-(u-\widetilde k_j)_+(y,\tau)|^p}{|x-y|^{N+sp}}\,dxdyd\tau
		\nonumber\\
	   \leq&\frac{\gamma2^{j(N+p+r)}}{(1-\sigma)^{N+p}\sigma^{sp}\theta}
		k^{-(r-q-1)}\iint_{Q_j}(u-k_j)_+^r\,dxd\tau.
	\end{align}

	Let $m=p\frac{N+s(q+1)}{N}$. Applying Lemma~\ref{lem-2-2} with \(h=q+1\) to
	\((u-k_{j+1})_+\) in the cylinder \(\widetilde Q_j\),
	we obtain
	\begin{align*}
		&\iint_{Q_{j+1}}(u-k_{j+1})_+^m\,dxd\tau
		\leq \iint_{\widetilde{Q}_j}(u-k_{j+1})_+^m\,dxd\tau\\
		\leq&\gamma\bigg[ \widetilde{\rho_j}^{sp}\int_{-\widetilde{\theta}_j}^{0} \int_{\widetilde{K}_j } \int_{\widetilde{K}_j } \frac{|(u-k_{j+1})_+(x,\tau)-(u-k_{j+1})_+(y,\tau)|^p}{|x - y|^{N+sp}}\, dxdyd\tau\\
		&\quad+ \int_{-\widetilde{\theta}_j}^{0} \int_{\widetilde{K}_j } |(u-k_{j+1})_+|^p\, dx d\tau\bigg]\times \left[\operatorname*{ess \sup}_{-\widetilde{\theta}_j<\tau<0} \mint_{\widetilde{K}_j \times\{\tau\}}(u-k_{j+1})_+ ^{q+1}\,dx\right]^{\frac{sp}{N}}\nonumber\\
		\leq& \frac{ \gamma 2^{j\left[N+p+r+(1-q)_+ \right]\frac{N+sp}{N}}}{(1-\sigma)^{\frac{(N+p)^2}{N}}\sigma^{\frac{sp(N+sp)}{N}} k^{(r-q-1) \frac{N+sp}{N}}}\left(\frac{1}{\theta^{\frac{N+sp}{N}}}+\frac{k^{p-q-1}}{\theta^{\frac{sp}{N}}\rho^{sp}}\right)\left[\iint_{Q_j}(u-k_j)_+^r\,dxd\tau\right]^{\frac{N+sp}{N}},
	\end{align*}
	where we used \eqref{3.19} to get the last line. The above estimate together with \eqref{choice of k1} gives
	\begin{align}
		\label{3.20}
		& \mint\!\!\mint_{Q_{j+1}}\left(u-k_{j+1}\right)_+^m\,dxd\tau\nonumber\\
		\leq&  \frac{ \gamma 2^{j\left[N+p+r+(1-q)_+ \right]\frac{N+sp}{N}}}{(1-\sigma)^{\frac{(N+p)^2}{N}}\sigma^{\frac{sp(N+sp)}{N}} k^{(r-q-1) \frac{N+sp}{N}}} \frac{\rho^{sp}}{\theta}\left[\mint\!\!\mint_{Q_j}(u-k_j)_+^r\,dxd\tau\right]^{\frac{N+sp}{N}},
	\end{align}
	where $\gamma$ depends only on $N,p,s,q,r,\Lambda$.
	
	Note that inequality \eqref{3.20} holds for every $r\geq q+1$. Hereafter, we distinguish three cases: $q+1\leq r\leq m$; $r\leq m$ and $r<q+1$; $r>m$ to discuss the local boundedness of weak subsolutions to \eqref{1.1}.

	\subsubsection{The case $q+1\leq r\leq m$}
	Set
	\begin{align*}
	Y_j=\mint\!\!\mint_{Q_j}(u-k_j)_+^r\,dxd\tau.
	\end{align*}
	By H\"older's inequality and \eqref{3.14}, there exists a constant $\gamma$ depending only on the \textit{data} and $r$ such that
		\begin{align*}
		Y_{j+1}&\leq \frac{1}{|Q_{j+1}|}\left[\iint_{Q_{j+1}}(u-k_{j+1})_+^m\,dxd\tau\right]^{\frac{r}{m}}|\widetilde{A}_j|^{1-\frac{r}{m}}\\
		&\leq \gamma\left[\mint\!\!\mint_{Q_{j+1}}(u-k_{j+1})_+^m\,dxd\tau\right]^{\frac{r}{m}}\left[\frac{2^{(j+2)r}}{k^r} \mint\!\!\mint_{Q_j}(u-k_j)_+^r\,dxd\tau\right]^{1-\frac{r}{m}}.
	\end{align*}
	Based on the definition $\lambda_r=N(p-q-1)+rsp$, it reads from \eqref{3.20} that
	\begin{align*}
		Y_{j+1}\leq \frac{\gamma b^j}{(1-\sigma)^{\frac{ r(N+p)^2}{N m}}\sigma^{\frac{rsp(N+sp)}{Nm}}}\left(\frac{\rho^{sp}}{\theta}\right)^{\frac{r}{m}} k^{-\frac{r \lambda_r}{N m}} Y_j^{1+\frac{rsp}{N m}},
	\end{align*}
	where $b>1$ depends only on the \textit{data} and $r$.
	Since $k_0(\tau)=\ell(\tau)\geq0$, we have
	\begin{align*}
		Y_0\leq\mint\!\!\mint_{Q_{0}} u^r\,dxd\tau.
	\end{align*}
	Lemma~\ref{lem-2-5} gives $Y_j\to0$ if in addition to
	\eqref{choice of k1} and \eqref{choice of k2}, we choose
	\begin{align}
		\label{3.21}
		k\geq
		\frac{\gamma}{[\sigma(1-\sigma)]^{\frac{(N+p)^2}{\lambda_r}}}\left(\frac{\rho^{sp}}\theta\right)^{\frac{N}{\lambda_r}}\left(\mint\!\!\mint_{Q_0}u^r\,dxd\tau\right)^{\frac{sp}{\lambda_r}}.
	\end{align}
	Since
	\begin{align*}
	k_j(\tau)\rightarrow k+\ell(\tau)\leq\frac{5k}{4},
	\end{align*}
	the convergence $Y_j\to0$ and the choice of $k$ yield, for
	$q\geq1$,
	\begin{align}
		\label{3.22}
		\operatorname*{ess\,sup}_{Q_{\sigma\rho,\sigma\theta}}u
		&\leq\frac{\gamma}{[\sigma(1-\sigma)]^{\frac{(N+p)^2}{\lambda_r}}}
		\left(\frac{\rho^{sp}}{\theta}\right)^{\frac{N}{\lambda_r}}
		\left(\mint\!\!\mint_{Q_0}u^r\,dxd\tau\right)^{\frac{sp}{\lambda_r}}
		\nonumber\\
		&\quad+
		\gamma\left(\frac{\theta}{\rho^{sp}}\right)^{\frac{1}{q+1-p}}
		+\gamma\left(\frac{\theta}{\rho^{sp}}
		\|u\|_{\infty,Q_{0}}^{p-1}\right)^{
		\frac1q}
		\nonumber\\
		&\quad+
		\frac{\gamma}{(1-\sigma)^{\frac{N+sp}{q}}}
		\mathcal T_+(u;0,\rho;-\theta,0)^{\frac1q}.
	\end{align}
	For $0<q<1$, the same estimate holds with the last term
	on the right-hand side replaced by
	\begin{align*}
	\frac{\gamma}{(1-\sigma)^{N+sp}}\|u\|_{\infty,Q_{0}}^{1-q}\mathcal T_+(u;0,\rho;-\theta,0).
	\end{align*}
	
	We next remove the dependence on $\|u\|_{\infty,Q_{0}}$. Set
	\begin{align*}
	M_\sigma:=\operatorname*{ess\,sup}_{Q_{\sigma\rho,\sigma\theta}}u.
	\end{align*}
	For
	\begin{align*}
	\frac12\leq\sigma_1<\sigma_2\leq\frac34,
	\end{align*}
	apply \eqref{3.22} in $Q_{\sigma_2\rho,\sigma_2\theta}$ with $\sigma=\sigma_1/\sigma_2$. Since $(p-1)/q<1$ and $1-q\in(0,1)$ for $0<q<1$, we use Young's inequality to derive
	\begin{align}
		\label{3.23}
		M_{\sigma_1}&\leq\frac12M_{\sigma_2}+\frac{\gamma}{(\sigma_2-\sigma_1)^\kappa}
		\Bigg[\left(\frac{\rho^{sp}}{\theta}\right)^{\frac{N}{\lambda_r}}
		\left(\mint\!\!\mint_{Q_{\rho,\theta}}u^r\,dxd\tau\right)^{\frac{sp}{\lambda_r}}
		\nonumber\\
		&\hspace{2cm}+\left(\frac{\theta}{\rho^{sp}}\right)^{\frac{1}{q+1-p}}
		+\mathcal T_+(u;0,\rho/2;-\theta,0)^{\frac1q}\Bigg],
	\end{align}
	where $\gamma,\kappa>0$ depend only on $N,p,s,q,\Lambda, r$.
	Since $M_{3/4}<\infty$, an application of Lemma~\ref{lem-2-6} on $[1/2,3/4]$ yields
	the desired result.
	
	\subsubsection{The case $r\leq m$ and $r<q+1$}
	Since
	\begin{align*}
	\lambda_{q+1}=\lambda_r+sp(q+1-r)>0,
	\end{align*}
	we have $q+1<m$. Then we can apply estimate \eqref{3.22} with exponent $q+1$. For $q\geq1$, applying it in $Q_{\sigma_2\rho,\sigma_2\theta}$ with $\sigma=\sigma_1/\sigma_2$ gives
	\begin{align}
		\label{3.24}
		M_{\sigma_1}
		&\leq\frac{\gamma}{(\sigma_2-\sigma_1)^\kappa}\left(\frac{\rho^{sp}}{\theta}\right)^{\frac{N}{\lambda_{q+1}}}\left(\mint\!\!\mint_{Q_{\sigma_2\rho,\sigma_2\theta}}
		u^{q+1}\,dxd\tau\right)^{\frac{sp}{\lambda_{q+1}}}\nonumber\\
		&\quad+\frac{\gamma}{(\sigma_2-\sigma_1)^\kappa}\left[\left(\frac{\theta}{\rho^{sp}}\right)^{\frac{1}{q+1-p}}+\left(\frac{\theta}{\rho^{sp}}\right)^{\frac1q}M_{\sigma_2}^{\frac{p-1}{q}}\right]\nonumber\\
		&\quad+\frac{\gamma}{(\sigma_2-\sigma_1)^\kappa}\mathcal T_+^{\frac1q}(u;0,\rho/2;-\theta,0),
	\end{align}
	where $\gamma,\kappa>0$ depend only on the \textit{data}.
	For $0<q<1$, the corresponding estimate is obtained by replacing the last term
	on the right-hand side with
	\begin{align*}
	\frac{\gamma}{(\sigma_2-\sigma_1)^\kappa}M_{\sigma_2}^{1-q}
	\mathcal T_+(u;0,\rho/2;-\theta,0).
	\end{align*}
	Using Young's inequality, the first term in \eqref{3.24} is bounded by
\begin{align*}
	&\frac{\gamma}{(\sigma_2-\sigma_1)^\kappa}
	M_{\sigma_2}^{1-\frac{\lambda_r}{\lambda_{q+1}}}
	\left(\frac{\rho^{sp}}{\theta}\right)^{\frac{N}{\lambda_{q+1}}}
	\left(\mint\!\!\mint_{Q_{\rho,\theta}}u^r\,dxd\tau\right)^{\frac{sp}{\lambda_{q+1}}}\\
    \leq&
	\frac14M_{\sigma_2}+\frac{\gamma}{(\sigma_2-\sigma_1)^{\kappa\frac{\lambda_{q+1}}{\lambda_r}}}\left(\frac{\rho^{sp}}{\theta}\right)^{\frac{N}{\lambda_r}}\left(\mint\!\!\mint_{Q_{\rho,\theta}}u^r\,dxd\tau\right)^{\frac{sp}{\lambda_r}}.
\end{align*}
Treating the remaining terms as in the preceding case gives
\eqref{3.23}. Since $M_{3/4}<\infty$, Lemma~\ref{lem-2-6} yields \eqref{1.7} in this case.
	
\subsubsection{The case $r>m$}
According to the assumption $\lambda_r>0$, we can see $r>q+1$ from $r>m$. If not, there will hold $0<\lambda_r\leq\lambda_{q+1}=N(m-q-1)$, which implies that $q+1<m$, and this would yield a contradiction. In this case, we use the local boundedness assumed in Theorem~\ref{thm-1-3}.
By \eqref{3.20} and the definition of $Y_j$, we obtain
		\begin{align*}
		Y_{j+1}  =&\mint\!\!\mint_{Q_{j+1}}(u-k_{j+1})_+^r\,dxd\tau  \\
		\leq&\|u\|_{\infty, Q_{0}}^{r-m} \mint\!\!\mint_{Q_{j+1}}(u-k_{j+1})_+^m\,dxd\tau \\
		\leq& \gamma\|u\|_{\infty, Q_{0}}^{r-m} \frac{\rho^{sp}}{\theta} \frac{b^j}{[(1-\sigma)\sigma]^{\frac{(N+p)^2}{N}} k^{(r-q-1) \frac{N+sp}{N}}}Y_j^{1+\frac{sp}{N}},
	\end{align*}
	where $b>1$ depends only on the \textit{data} and $r$. Since $k_0(\tau)=\ell(\tau)\geq0$, we have
	\begin{align*}
	Y_0
	=\mint\!\!\mint_{Q_0}(u-\ell(\tau))_+^r\,dxd\tau
	\leq
	\mint\!\!\mint_{Q_0}u^r\,dxd\tau.
	\end{align*}
	Utilizing Lemma~\ref{lem-2-5}, it holds that $Y_j\to0$ as $j\to\infty$ provided that
	\begin{align*}
		\mint\!\!\mint_{Q_0} u^r\,dxd\tau \leq \gamma^{-1}\|u\|_{\infty, Q_0}^{-(r-m) \frac{N}{sp}}[(1-\sigma)\sigma]^{\frac{(N+p)^2}{sp}}\left(\frac{\theta}{\rho^{sp}}\right)^{\frac{N}{sp}} k^{\frac{(r-q-1)(N+sp)}{sp}},
	\end{align*}
	Thus, in addition to \eqref{choice of k1} and \eqref{choice of k2}, it suffices to choose
	\begin{align}
		\label{3.25}
		k\geq\frac{\gamma\|u\|_{\infty,Q_0}^{\frac{N(r-m)}{(N+sp)(r-q-1)}}}
		{[\sigma(1-\sigma)]^{\frac{(N+p)^2}{(N+sp)(r-q-1)}}}
		\left(\frac{\rho^{sp}}{\theta}\right)^{\frac{N}{(N+sp)(r-q-1)}}
		\left(\mint\!\!\mint_{Q_0}u^r\,dxd\tau\right)^{\frac{sp}{(N+sp)(r-q-1)}}.
	\end{align}
	Taking the choices of $k$ in \eqref{choice of k1}, \eqref{choice of k2} and \eqref{3.25} into account, we deduce from $Y_j\to 0$ that, for $q\geq 1$,
	\begin{align*}
		\operatorname*{ess\,sup}_{Q_{\sigma\rho,\sigma\theta}}u
		&\leq
		\frac{\gamma\|u\|_{\infty,Q_0}^{\frac{N(r-m)}{(N+sp)(r-q-1)}}}{	[\sigma(1-\sigma)]^{\frac{(N+p)^2}{(N+sp)(r-q-1)}}}
		\left(\frac{\rho^{sp}}{\theta}\right)^{\frac{N}{(N+sp)(r-q-1)}}
		\left(\mint\!\!\mint_{Q_0}u^r\,dxd\tau\right)^{\frac{sp}{(N+sp)(r-q-1)}}\\
		&\quad+\gamma\left(\frac{\theta}{\rho^{sp}}\right)^{\frac{1}{q+1-p}}
		+\gamma\left(\frac{\theta}{\rho^{sp}}\|u\|_{\infty,Q_{\rho,\theta}}^{p-1}
		\right)^{\frac1q}+\frac{\gamma}{(1-\sigma)^{\frac{N+sp}{q}}}
		\mathcal T_+^{\frac1q}(u;0,\rho;-\theta,0),
	\end{align*}
where $\gamma$ depends only on the \textit{data} and $r$.
	For $0<q<1$, the last term is replaced by
	\begin{align*}
	\frac{\gamma}{(1-\sigma)^{N+sp}}\|u\|_{\infty,Q_{\rho,\theta}}^{1-q}\mathcal T_+(u;0,\rho;-\theta,0).
	\end{align*}
	We may proceed as in the previous case and apply Lemma \ref{lem-2-6} to conclude the proof.
	
	\section{Integral-type Harnack inequality}
	\label{sec4}
	In this section, we establish Theorem \ref{thm-1-7} by combining Theorem \ref{thm-1-3} with the  integral-type Harnack inequality in Theorem \ref{thm-1-6}. For convenience, we assume that the interval under consideration is $[0,\theta]$ by a time translation.
	To proceed, we require the following technical lemma, which provides a refined energy estimate.
	\begin{lemma}
		\label{lem-4-1}
		Let $0<p-1<q$ and $u$ be a weak solution to \eqref{1.1} satisfying
		$u\geq 0$ in
		$K_{\rho}(x_0)\times(0,\theta]\Subset E_T$.
		Let $0<\sigma<1$, $\sigma'=(1+\sigma)/2$, and let
		$\xi\in C_c^1\left(K_{\sigma'\rho}(x_0)\right)$ be a cutoff function
		satisfying
		\begin{align*}
			\xi=1 \quad \text{in } K_{\sigma\rho}(x_0),\qquad 0\leq\xi\leq 1,
			\qquad |\nabla \xi|\leq\frac{\gamma}{(1-\sigma)\rho}\quad \text{in } \mathbb{R}^N.
		\end{align*}
		For $\tau\in(0,\theta)$, set
		\begin{align*}
			A_\tau:=\{(x,y)\in K_{\sigma'\rho}(x_0)\times K_{\sigma'\rho}(x_0):u(x,\tau)>u(y,\tau),\,\xi(x)>\xi(y)\},
		\end{align*}
		and
		\begin{align*}
			B_\tau:=\{(x,y)\in K_{\sigma'\rho}(x_0)\times K_{\sigma'\rho}(x_0):
			u(x,\tau)>u(y,\tau),\,\xi(x)<\xi(y)\}.
		\end{align*}
		Let
		\begin{align}
			\label{4.1}
			\kappa:=
			\max\left\{\left(\frac{\theta}{\rho^{sp}}\right)^{\frac{1}{q+1-p}},\,
			\mathcal T_-(u;x_0,\frac{\rho}{2};0,\theta)^{\frac1q}\right\}.
		\end{align}
		Then there exists a constant $\gamma>0$ depending only on
		$N,p,s,q,\Lambda$ such that
		\begin{align*}
			&\int_0^\theta\tau^{\frac{1}{p}}\iint_{A_\tau}
			\frac{(u(x,\tau)-u(y,\tau))^p}{|x-y|^{N+sp}}(u(x,\tau)+\kappa)^{-\frac{q+1}{p}}
			\xi^p(x)\,dxdyd\tau\\
			&\quad+
			\int_0^\theta\tau^{\frac{1}{p}}\iint_{B_\tau}\frac{(u(x,\tau)-u(y,\tau))^p}{|x-y|^{N+sp}}
			(u(x,\tau)+\kappa)^{-\frac{q+1}{p}}\xi^p(y)\,dxdyd\tau\\
			&\leq
			\frac{\gamma\rho^s}{(1-\sigma)^{N+p}}\left[
			\sup_{0<\tau<\theta}\int_{K_\rho(x_0)}u^q(x,\tau)\,dx+\kappa^q\rho^N
			\right]^{\frac{(q+1)(p-1)}{pq}}\left(\frac{\theta}{\rho^{\lambda}}\right)^{\frac{1}{p}},
		\end{align*}
		where
		\begin{align*}
		\lambda:=\frac{\lambda_q}{q}=\frac{N}{q}(p-q-1)+sp.
		\end{align*}
	\end{lemma}
	
	\begin{proof}
		We choose
		\begin{align*}
			\varphi_h(x,\tau)=\tau^{\frac{1}{p}}\left([\![u]\!]_ {\bar{h}}+\kappa\right)^{-\frac{q+1-p}{p}} \xi^p(x) \psi_{\varepsilon}(\tau) 
		\end{align*}
		as a test function in the weak formulation \eqref{1.5}. Here, $[\![u]\!]_{\bar h}$ denotes the
		backward exponential mollification on $(0,\theta)$, with the upper
		integration limit $T$ replaced by $\theta$. Moreover, $\xi\in C_0^1\left(K_{\frac{\rho(1+\sigma)}{2}}\right)$ satisfies $\xi=1$ in $K_{\sigma\rho}$ and $|\nabla\xi|\leq \frac{\gamma}{(1-\sigma)\rho}$, for $\varepsilon>0$, $\psi_\varepsilon(\tau)$ is a Lipschitz function such that $\psi_{\varepsilon}=1$ in $(\varepsilon, \theta-\varepsilon), \psi_{\varepsilon}=0$ outside $(0, \theta)$, and it is linearly interpolated otherwise. Then we obtain
		\begin{align}
			\label{4.2}
			&\iint_{E_T}-u^q \partial_\tau \varphi_h\,dxd\tau+\frac12
			\int_{0}^{T}\int_{\mathbb{R}^N}\int_{\mathbb{R}^N}U(x,y,\tau)(\varphi_h(x,\tau)-\varphi_h(y,\tau))\,d\mu d\tau=0.
		\end{align}
		We begin by estimating the first term in \eqref{4.2}. Indeed,
		\begin{align}
			\label{4.3}
			&-\iint_{E_T} u^q \partial_\tau \varphi_h\,dxd\tau\nonumber\\
			=&-\iint_{E_T} [\![u]\!]_{\bar{h}}^q \partial_\tau \varphi_h\,dxd\tau+\iint_{E_T}\left([\![u]\!]_{\bar{h}}^q-u^q\right) \partial_\tau \varphi_h\,dxd\tau\nonumber\\
			=&\iint_{E_T} \partial_\tau[\![u]\!]_{\bar{h}}^q \varphi_h\,dxd\tau+\iint_{E_T}\left([\![u]\!]_{\bar{h}}^q-u^q\right)\left([\![u]\!]_{\bar{h}}+\kappa\right)^{-\frac{q+1-p}{p}} \xi^p(x) \partial_\tau(\tau^{\frac{1}{p}} \psi_{\varepsilon}(\tau))\,dxd\tau\nonumber\\
			&-\frac{q+1-p}{p} \iint_{E_T}\left([\![u]\!]_{\bar{h}}+\kappa\right)^{-\frac{q+1}{p}}\left([\![u]\!]_{\bar{h}}^q-u^q\right) \frac{1}{h}\left([\![u]\!]_{\bar{h}}-u\right) \tau^{\frac{1}{p}} \xi^p(x) \psi_{\varepsilon}(\tau)\,dxd\tau,
		\end{align}
		where we used Lemma \ref{lem-2-8} (ii) to obtain the last line. Notice that the second term tends to zero as $h\rightarrow0$ by Lemma \ref{lem-2-8} (i). Moreover, the third term is nonpositive owing to the monotonicity of the function $s\mapsto s^q$. It therefore remains to estimate the first term. Integrating by parts and then letting first $h\to0$
		and subsequently $\varepsilon\to0$, we derive
		\begin{align}
			\label{4.4}
			&\iint_{E_T} \partial_\tau[\![u]\!]_{\bar{h}}^q \varphi_h\,dxd\tau\nonumber\\
			= & q \iint_{E_T} \xi^p(x) \tau^{\frac{1}{p}} \psi_{\varepsilon}(\tau) \partial_\tau\left[\int_0^{[\![u]\!]_{\bar{h}}} s^{q-1}(s+\kappa)^{-\frac{q+1-p}{p}} \,ds\right]dxd\tau\nonumber\\
			= & -q \iint_{E_T} \xi^p(x) \tau^{\frac{1}{p}} \partial_\tau \psi_{\varepsilon}(\tau) \int_0^{[\![u]\!]_{\bar{h}}} s^{q-1}(s+\kappa)^{-\frac{q+1-p}{p}}\,dsdxd\tau \nonumber\\
			& -q \iint_{E_T} \xi^p(x) \partial_\tau \tau^{\frac{1}{p}} \psi_{\varepsilon}(\tau) \int_0^{[\![u]\!]_{\bar{h}}} s^{q-1}(s+\kappa)^{-\frac{q+1-p}{p}} \,dsdxd\tau\nonumber\\
			\rightarrow & \left.q \int_{K_{\rho}} \xi^p(x) \tau^{\frac{1}{p}} \int_0^u s^{q-1}(s+\kappa)^{-\frac{q+1-p}{p}} \,dsdx\right|_0 ^\theta \nonumber\\
			&-q \iint_{K_{\rho} \times(0, \theta)} \xi^p(x) \partial_\tau \tau^{\frac{1}{p}} \int_0^u s^{q-1}(s+\kappa)^{-\frac{q+1-p}{p}}\,dsdxd\tau \nonumber\\			
			= &q \theta^{\frac{1}{p}} \int_{K_{\rho} \times\{\theta\}} \xi^p(x) \int_0^u s^{q-1}(s+\kappa)^{-\frac{q+1-p}{p}} \,dsdx\nonumber\\
			& -q \iint_{K_{\rho} \times(0, \theta)} \xi^p(x) \partial_\tau \tau^{\frac{1}{p}} \int_0^u s^{q-1}(s+\kappa)^{-\frac{q+1-p}{p}} \,dsdxd\tau.			
		\end{align}	
		For the nonlocal term, we first let $h\rightarrow0$. The passage to the limit can be justified by combining Lemma \ref{lem-2-8}(i) with the argument in \cite[Appendix B]{L24}. Then, letting $\varepsilon\rightarrow0$, we obtain
		\begin{align}
			\label{4.5}
			&\int_{0}^{T}\int_{\mathbb{R}^N}\int_{\mathbb{R}^N}|u(x,\tau)-u(y,\tau)|^{p-2}(u(x,\tau)-u(y,\tau))(\varphi_h(x,\tau)-\varphi_h(y,\tau))\,d\mu d\tau\nonumber\\
			\rightarrow& \int_{0}^{\theta}\int_{\mathbb{R}^N}\int_{\mathbb{R}^N} |u(x,\tau)-u(y,\tau)|^{p-2}(u(x,\tau)-u(y,\tau))(\varphi(x,\tau)-\varphi(y,\tau))\,d\mu d\tau.
		\end{align}
		Collecting the estimates in \eqref{4.3}--\eqref{4.5}, we obtain
		\begin{align}
			\label{4.6}
			0\leq & q \theta^{\frac{1}{p}} \int_{K_{\rho} \times\{\theta\}} \xi^p(x) \int_0^u s^{q-1}(s+\kappa)^{-\frac{q+1-p}{p}} \,dsdx\nonumber\\
			&+\frac12\int_{0}^{\theta}\int_{\mathbb{R}^N}\int_{\mathbb{R}^N}U(x,y,\tau) (\varphi(x,\tau)-\varphi(y,\tau))\,d\mu d\tau.
		\end{align}
		
		At this stage, we leave the first term in \eqref{4.6} unchanged and focus on estimating the second term. By the symmetry of the kernel and the integrand with respect to $x$ and $y$, we may restrict our attention to the case $u(x,\tau)> u(y,\tau)$ for some $\tau\in(0,\theta)$. The opposite case can be handled in exactly the same way by interchanging the roles of $x$ and $y$.  
		
		Here, we first consider the case when $u(x,\tau)> u(y,\tau)\geq 0$. We distinguish the cases $\xi(x)>\xi(y)$ and $\xi(x)\leq\xi(y)$. If $\xi(x)>\xi(y)$, then
		\begin{align}
			\label{4.7}
			&U(x,y,\tau)(\varphi(x,\tau)-\varphi(y,\tau))\nonumber\\
			=&\tau^{\frac{1}{p}}(u(x,\tau)-u(y,\tau))^{p-1} \left[ (u(x, \tau) + \kappa)^{\frac{p-q-1}{p}} \xi^p(x) - (u(y, \tau) + \kappa)^{\frac{p-q-1}{p}} \xi^p(y) \right]\nonumber\\
			\leq&\tau^{\frac{1}{p}} (u(x,\tau)-u(y,\tau))^{p-1}\left[ (u(x, \tau) + \kappa)^{\frac{p-q-1}{p}} - (u(y, \tau) + \kappa)^{\frac{p-q-1}{p}}  \right]\xi^p(x)\nonumber \\
			&+\tau^{\frac{1}{p}}(u(x,\tau)-u(y,\tau))^{p-1}(u(y, \tau) + \kappa)^{\frac{p-q-1}{p}}\left(\xi^p(x)-\xi^p(y)\right).
		\end{align}	
		Since $0<p-1<q$, the mean value theorem yields
		\begin{align}
			\label{4.8}
			&(u(x,\tau)-u(y,\tau))^{p-1}\left[ (u(x, \tau) + \kappa)^{\frac{p-q-1}{p}} - (u(y, \tau) + \kappa)^{\frac{p-q-1}{p}}  \right]\xi^p(x)\nonumber\\
			\leq& -\gamma(p,q)(u(x,\tau)-u(y,\tau))^p (u(x, \tau) + \kappa)^{-\frac{q+1}{p}} \xi^p(x).
		\end{align}
		Moreover, to estimate the second term in \eqref{4.7}, we apply the elementary inequality \cite[(2.4)]{L24} with $a=\xi(x)>\xi(y)=b\geq 0$. Here, $\delta\in(0,1)$ is a parameter to be specified later, and $\varepsilon\in(0,1)$ is defined by
		$$\varepsilon = \delta \frac{u(x, \tau) - u(y, \tau)}{u(y, \tau) + \kappa} \left[ \frac{u(y, \tau) + \kappa}{u(x, \tau) + \kappa} \right]^{\frac{q+1}{p}} < 1.$$
	It follows that
		\begin{align*}
			\xi^p(x) - \xi^p(y) \leq& \delta \frac{u(x, \tau) - u(y, \tau)}{u(y, \tau) + \kappa} \left[ \frac{u(y, \tau) + \kappa}{u(x, \tau) + \kappa} \right]^{\frac{q+1}{p}} \xi^p(x) \\
			&+ \frac{\gamma}{\delta^{p-1}} \left[ \frac{u(y, \tau) + \kappa}{u(x, \tau) - u(y, \tau)} \right]^{p-1} \left[ \frac{u(x, \tau) + \kappa}{u(y, \tau) + \kappa} \right]^{\frac{(q+1)(p-1)}{p}} (\xi(x) - \xi(y))^p.
		\end{align*}
		Subsequently,
		\begin{align}
			\label{4.9}
			&(u(x,\tau)-u(y,\tau))^{p-1}(u(y, \tau) + \kappa)^{\frac{p-q-1}{p}}\left(\xi^p(x)-\xi^p(y)\right)\nonumber\\
			\leq&\gamma \delta (u(x, \tau) - u(y, \tau))^p(u(x, \tau) + \kappa)^{-\frac{q+1}{p}} \xi^p(x) \nonumber\\
			&+ \frac{\gamma}{\delta^{p-1} \kappa^{q+1-p}}(u(x, \tau) + \kappa)^{\frac{(q+1)(p-1)}{p}} (\xi(x) - \xi(y))^p.
		\end{align}
		By combining \eqref{4.7}--\eqref{4.9}, and choosing $\delta\in(0,1)$ small enough, we arrive at 
		\begin{align*}
			U(x,y,\tau)(\varphi(x,\tau)-\varphi(y,\tau))\leq&-\gamma \tau^{\frac{1}{p}} (u(x,\tau)-u(y,\tau))^p(u(x, \tau) + \kappa)^{-\frac{q+1}{p}}\xi^p(x) \\
			&+ \frac{\gamma \tau^{\frac{1}{p}}}{ \kappa^{q+1-p}}(u(x, \tau) + \kappa)^{\frac{(q+1)(p-1)}{p}} (\xi(x) - \xi(y))^p.
		\end{align*}
		We now consider the case $\xi(x)<\xi(y)$. Employing the mean value theorem, we have
		\begin{align*}
			\varphi(x,\tau)-\varphi(y,\tau)
			&=\tau^{\frac1p}\left[(u(x,\tau)+\kappa)^{\frac{p-q-1}{p}}\xi^p(x)-(u(y,\tau)+\kappa)^{\frac{p-q-1}{p}}\xi^p(y)\right]\\
			&\leq\tau^{\frac1p}\left[(u(x,\tau)+\kappa)^{\frac{p-q-1}{p}}-(u(y,\tau)+\kappa)^{\frac{p-q-1}{p}}\right]\xi^p(y)\\
			&\leq-\gamma\tau^{\frac1p}(u(x,\tau)-u(y,\tau))(u(x,\tau)+\kappa)^{-\frac{q+1}{p}}\xi^p(y).
		\end{align*}
           Thus there holds
		\begin{align}
			\label{4.10}
			&U(x,y,\tau)\left(\varphi(x,\tau)-\varphi(y,\tau)\right)\leq
			-\gamma\tau^{\frac1p}(u(x,\tau)-u(y,\tau))^p
			(u(x,\tau)+\kappa)^{-\frac{q+1}{p}}\xi^p(y).
		\end{align}
		When $\xi(x)=\xi(y)$, the corresponding contribution is nonpositive and may be discarded. 
		We next turn to the case $u(x,\tau)>u(y,\tau)$ and $u(y,\tau)<0$. In this case, it follows immediately that $y\in \mathbb{R}^N\backslash K_\rho$, and hence $\xi(y)=0$. Consequently,
		\begin{align*}
			&U(x,y,\tau)(\varphi(x,\tau)-\varphi(y,\tau))\leq\gamma \tau^{\frac{1}{p}} (u^{p-1}(x,\tau)+u_-^{p-1}(y,\tau))(u(x, \tau) + \kappa)^{\frac{p-q-1}{p}} \xi^p(x).
		\end{align*}
		Gathering the estimates in the above two cases together, there holds that
		\begin{align}
			\label{4.11}
			0\leq & q \theta^{\frac{1}{p}} \int_{K_{\rho} \times\{\theta\}} \xi^p(x) \int_0^u s^{q-1}(s+\kappa)^{-\frac{q+1-p}{p}} \,dsdx\nonumber\\
			&-\gamma\int_0^\theta\tau^{\frac{1}{p}}\iint_{A_\tau}\frac{(u(x,\tau)-u(y,\tau))^p}{|x-y|^{N+sp}} (u(x, \tau) + \kappa)^{-\frac{q+1}{p}} \xi^p(x)\,dxdyd\tau\nonumber\\
			&-\gamma\int_0^\theta\tau^{\frac{1}{p}}\iint_{B_\tau}\frac{(u(x,\tau)-u(y,\tau))^p}{|x-y|^{N+sp}} (u(x, \tau) + \kappa)^{-\frac{q+1}{p}} \xi^p(y)\,dxdyd\tau\nonumber\\
			&+\frac{\gamma}{\kappa^{q+1-p}}\int_0^\theta\tau^{\frac{1}{p}}\int_{K_{\sigma'\rho}}\int_{ K_{\rho}}(u(x, \tau) + \kappa)^{\frac{(q+1)(p-1)}{p}} \frac{|\xi(x) - \xi(y)|^p}{|x-y|^{N+sp}}\,dydxd\tau\nonumber\\
			&+\frac{\gamma}{ \kappa^{q+1-p}}\int_0^\theta\tau^{\frac{1}{p}}\int_{K_{\sigma'\rho}}\int_{ \mathbb{R}^N\backslash K_{\rho}}(u(x, \tau) + \kappa)^{\frac{(q+1)(p-1)}{p}} \frac{|\xi(x) - \xi(y)|^p}{|x-y|^{N+sp}}\,dydxd\tau\nonumber\\
			&+\gamma\int_0^\theta\tau^{\frac{1}{p}}\int_{K_{\sigma'\rho}}\int_{\mathbb{R}^N\backslash K_{\rho}}\frac{u^{p-1}(x,\tau)+u_-^{p-1}(y,\tau)}{|x-y|^{N+sp}}(u(x, \tau) + \kappa)^{\frac{p-q-1}{p}} \xi^p(x)\,dydxd\tau\nonumber\\
			=:&I_1+I_2+I_3+I_4+I_5+I_6.
		\end{align}
		In what follows, we estimate $I_1$ and $I_4$--$I_6$ separately, while retaining $I_2$ and $I_3$.
		
		\textbf{Estimate of $I_1$:} From the hypothesis $0<p-1<q$, one can easily get
		\begin{align*}
			\frac{(p-1)(q+1)}{pq} \in(0,1) \quad \text { and } \quad \int_0^u s^{q-1}(s+\kappa)^{-\frac{q+1-p}{p}}\,ds \leq \gamma (u+\kappa)^{\frac{(p-1)(q+1)}{p}}.
		\end{align*}	
		Thus we apply H\"{o}lder's inequality to obtain
		\begin{align}
			\label{4.12}
			I_1&\leq\gamma \theta^{\frac{1}{p}}\sup_{0<\tau<\theta} \int_{K_{\rho} \times\{\tau\}} (u+\kappa)^{\frac{(p-1)(q+1)}{p}}\,dx \nonumber\\
			&\leq \gamma \rho^s\left(\frac{\theta}{\rho^\lambda}\right)^{\frac{1}{p}}\left(\sup_{0<\tau<\theta} \int_{K_{\rho} \times\{\tau\}} u^q\,dx+\kappa^q\rho^N\right)^{\frac{(p-1)(q+1)}{p q}}.
		\end{align}	
		
		\textbf{Estimate of $I_4$:} Applying H\"{o}lder's inequality and using the properties of $\xi$ together with the definition of $\kappa$, we get
		\begin{align}
			\label{4.13}
			I_4&\leq \frac{\gamma}{(1- \sigma)^p \rho^p \kappa^{q+1-p}} \int_0^\theta\tau^{\frac{1}{p}}\int_{K_{\sigma'\rho}} \int_{K_\rho} \frac{(u(x, \tau) + \kappa)^{\frac{(q+1)(p-1)}{p}}}{|x - y|^{N+sp-p}}\,dydxd\tau\nonumber\\
			&\leq \frac{\gamma \theta^{1+\frac{1}{p}}}{(1- \sigma)^p \rho^{sp} \kappa^{q+1-p}} \sup_{0<\tau<\theta} \int_{K_{\rho}\times \{\tau\}} (u(x, \tau) + \kappa)^{\frac{(q+1)(p-1)}{p}}\,dx\nonumber\\
			&\leq \frac{\gamma}{(1- \sigma)^p} \rho^s\left(\frac{\theta}{\rho^\lambda}\right)^{\frac{1}{p}}\left(\sup_{0<\tau<\theta} \int_{K_{\rho} \times\{\tau\}} u^q\,dx+\kappa^q\rho^N\right)^{\frac{(p-1)(q+1)}{p q}}.
		\end{align}
		
		\textbf{Estimate of $I_5$:} Notice that for all $x\in K_{\sigma'\rho}$, $y\in\mathbb{R}^N\backslash K_\rho$, there holds
		$$\frac{|y-x_0|}{|y-x|}\leq  1+\frac{|x-x_0|}{|y-x|}\leq\frac{\gamma(N)}{1-\sigma}.$$
		Thus we have 
		\begin{align}
			\label{4.14}
			I_5\leq & \frac{\gamma}{(1 - \sigma)^{N+sp} \kappa^{q+1-p}} \int_0^\theta\tau^{\frac{1}{p}} \int_{K_{\sigma'\rho}}\int_{\mathbb{R}^N\backslash K_\rho}\frac{(u(x, \tau) + \kappa)^{\frac{(q+1)(p-1)}{p}}}{|y - x_0|^{N+sp}} \, dydxd\tau\nonumber\\
			&\leq \frac{\gamma \theta^{1+\frac{1}{p}}}{(1- \sigma)^{N+sp} \rho^{sp} \kappa^{q+1-p}} \sup_{0<\tau<\theta} \int_{K_{\rho}\times \{\tau\}} (u(x, \tau) + \kappa)^{\frac{(q+1)(p-1)}{p}}\,dx\nonumber\\
			&\leq \frac{\gamma}{(1- \sigma)^{N+sp}} \rho^s\left(\frac{\theta}{\rho^\lambda}\right)^{\frac{1}{p}}\left(\sup_{0<\tau<\theta} \int_{K_{\rho} \times\{\tau\}} u^q\,dx+\kappa^q\rho^N\right)^{\frac{(p-1)(q+1)}{p q}}.
		\end{align}
		
		\textbf{Estimate of $I_6$:} We deduce that
		\begin{align}
			\label{4.15}
			I_6\leq&\frac{\gamma}{\kappa^{q+1-p}}\int_0^\theta\tau^{\frac{1}{p}} \int_{K_{\sigma'\rho}}\int_{\mathbb{R}^N\backslash K_\rho}\frac{(u(x,\tau)+\kappa)^{\frac{(q+1)(p-1)}{p}}}{|x-y|^{N+sp}}\,dydxd\tau\nonumber\\
			&+\frac{\gamma}{\kappa^q}\int_0^\theta\tau^{\frac{1}{p}} \int_{K_{\sigma'\rho}}\int_{\mathbb{R}^N\backslash K_\rho}\frac{(u(x,\tau)+\kappa)^{\frac{(q+1)(p-1)}{p}}}{|x-y|^{N+sp}}u_-^{p-1}(y,\tau)\,dydxd\tau\nonumber\\
			\leq &\frac{\gamma}{(1- \sigma)^{N+sp}} \rho^s\left(\frac{\theta}{\rho^\lambda}\right)^{\frac{1}{p}}\left(\sup_{0<\tau<\theta} \int_{K_{\rho} \times\{\tau\}} u^q\,dx+\kappa^q\rho^N\right)^{\frac{(p-1)(q+1)}{p q}}\nonumber\\
			&+\frac{\gamma \theta^{\frac{1}{p}}}{(1-\sigma)^{N+sp} \kappa^q}  \sup_{0<\tau<\theta}\int_{K_\rho\times\{\tau\}} (u(x, \tau) + \kappa)^{\frac{(p - 1)(q+1)}{p}}\, dx \nonumber\\
			&\quad\times \int_{0}^{\theta}\int_{\mathbb{R}^N\backslash K_\rho} \frac{u_-^{p - 1}(y, \tau)}{|y-x_0|^{N+sp}}\,dyd\tau\nonumber\\
			\leq& \frac{\gamma}{(1- \sigma)^{N+sp}} \rho^s\left(\frac{\theta}{\rho^\lambda}\right)^{\frac{1}{p}}\left(\sup_{0<\tau<\theta} \int_{K_{\rho} \times\{\tau\}} u^q\,dx+\kappa^q\rho^N\right)^{\frac{(p-1)(q+1)}{p q}},
		\end{align}
		where we employed the definition of $\mathcal{T}_-$ and $\kappa$.
		
		Substituting \eqref{4.12}--\eqref{4.15} into \eqref{4.11}, we obtain
		\begin{align*}
			&\int_0^\theta\tau^{\frac1p}\iint_{A_\tau}\frac{(u(x,\tau)-u(y,\tau))^p}{|x-y|^{N+sp}}
			(u(x,\tau)+\kappa)^{-\frac{q+1}{p}}\xi^p(x)\,dxdyd\tau\\
			&\quad+\int_0^\theta\tau^{\frac1p}\iint_{B_\tau}\frac{(u(x,\tau)-u(y,\tau))^p}
			{|x-y|^{N+sp}}(u(x,\tau)+\kappa)^{-\frac{q+1}{p}}\xi^p(y)\,dxdyd\tau\\
			&\leq\frac{\gamma\rho^s}{(1-\sigma)^{N+p}}\left(\sup_{0<\tau<\theta}
			\int_{K_\rho\times\{\tau\}}u^q\,dx+\kappa^q\rho^N\right)^{\frac{(p-1)(q+1)}{pq}}
			\left(\frac{\theta}{\rho^\lambda}\right)^{\frac1p}.
		\end{align*}
		This completes the proof.
	\end{proof}
	
	\begin{proof}[Proof of Theorem \ref{thm-1-6}]
		We perform an iteration argument by setting 
		\begin{align*}
			\rho_j=\sum_{n=0}^j \frac{\rho}{2^n}, \quad \widetilde{\rho}_j=\frac{3\rho_j+\rho_{j+1}}{4}, \quad \widehat{\rho}_j=\frac{\rho_j+\rho_{j+1}}{2}, \ \ j=0,1,2,\ldots,
		\end{align*}
		and
		\begin{align*}
			\quad K_j=K_{\rho_j}(x_0), \quad \widetilde{K}_j=K_{\widetilde{\rho}_j}(x_0), \quad \widehat{K}_j=K_{\widehat{\rho}_j}(x_0), \ \ j=0,1,2,\ldots.
		\end{align*}
		Consider the function $\xi_j\in C_c^\infty(\widetilde{K}_j)$ such that $0\leq\xi_j\leq 1$,
		$\xi_j=1$ in $K_j$, and $|\nabla\xi_j|\leq \gamma 2^{j}/\rho$. 
		Moreover, let $t_1, t_2\in[0,\theta]$ such that 
		\begin{align*}
			\int_{K_{2 \rho} }u^q(x,t_2)\,dx=\inf_{0<\tau<\theta} \int_{K_{2\rho}} u^q(x,\tau)\,dx=: A .
		\end{align*}
		Define 
		\begin{align*}
			\int_{K_j}u^q(x,t_1)\,dx=\sup_{0<\tau<\theta} \int_{K_j } u^q (x,\tau)\,dx	=:S_j.
		\end{align*}


		We first consider the case $0\leq t_1\leq t_2\leq\theta$. Testing \eqref{1.1} with $\xi_j^{p+1}$ over $(t_1,t_2)$, we obtain 
		\begin{align}
			\label{4.16}
			S_j\leq&\int_{\widetilde{K}_j } u^q(x,t_1) \xi_j^{p+1}\,dx\nonumber \\
			\leq&\int_{K_{2\rho}}u^q(x,t_2)\,dx+\frac12 \int_{t_1}^{t_2} \int_{\mathbb{R}^N} \int_{\mathbb{R}^N} U(x,y,\tau) (\xi_j^{p+1}(x) - \xi_j^{p+1}(y)) \, d\mu d\tau\nonumber\\
			\leq &\int_{K_{2\rho}}u^q(x,t_2)\,dx+\frac12 \int_{t_1}^{t_2} \int_{\widehat{K}_j}\int_{\widehat{K}_j} U(x,y,\tau)(\xi_j^{p+1}(x) - \xi_j^{p+1}(y))\,d\mu d\tau\nonumber\\
			&+ \gamma \int_{t_1}^{t_2} \int_{\widehat{K}_j}\int_{\mathbb{R}^N\backslash\widehat{K}_j}     \frac{u^{p-1}(x, \tau)}{|x - y|^{N+sp}} \xi_j^{p+1}(x) \,dxdyd\tau\nonumber \\
			&+ \gamma \int_{t_1}^{t_2} \int_{\widehat{K}_j}\int_{\mathbb{R}^N\backslash\widehat{K}_j}     \frac{u_-^{p-1}(y, \tau)}{|x - y|^{N+sp}} \xi_j^{p+1}(x) \,dxdyd\tau \nonumber\\
			=:&J_1+J_2+J_3+J_4.
		\end{align}
		In the following, we derive the estimate of nonlocal terms.
		
		\textbf{Estimate of $J_2$:} By symmetry, we discard the nonpositive contributions
		and then apply the upper bound of \eqref{kernel} to obtain
		\begin{align*}
			J_2 \leq \gamma \int_0^\theta \iint_{A_{j,\tau}} \frac{(u(x, \tau) - u(y, \tau))^{p-1}}{|x - y|^{N+sp}} (\xi_j^{p+1}(x) - \xi_j^{p+1}(y)) \, dx dy d\tau,
		\end{align*}
		where 
		\begin{align*}
			A_{j,\tau} =\left \{(x, y) \in \widehat{K}_j \times \widehat{K}_j : u(x, \tau) > u(y, \tau), \xi_j(x) > \xi_j(y)\right\}\text{  for all } \tau\in(0,\theta).
		\end{align*}
		Define
		\begin{align*}
			\kappa_j:=\max\left\{\left(\frac{\theta}{\rho_{j+1}^{sp}}\right)^{\frac{1}{q+1-p}},\,
			\mathcal T_-\left(u;x_0,\frac{\rho_{j+1}}2;0,\theta\right)^{\frac1q}\right\}.
		\end{align*}
		By H\"{o}lder's inequality, there holds that
		\begin{align}
			\label{4.17}
			J_2\leq& \frac{\gamma 2^j}{\rho} \int_0^\theta \iint_{A_{j,\tau}} \frac{(u(x, \tau) - u(y, \tau))^{p-1}}{|x - y|^{N+sp-1}} \xi_j^p(x) \,dxdy d\tau\nonumber\\
			\leq& \frac{\gamma 2^j}{\rho} \int_0^\theta \iint_{A_{j,\tau}} \left[ (u(x, \tau) - u(y, \tau))^{p-1} (u(x, \tau) + \kappa_j)^{\frac{(q+1)(1-p)}{p^2}} \tau^{\frac{p-1}{p^2}} \right]\nonumber\\
			&\quad\quad\quad\quad\quad\quad\times \left[ (u(x, \tau) + \kappa_j)^{\frac{(q+1)(p-1)}{p^2}} \tau^{\frac{1-p}{p^2}} |x - y| \right] \frac{\xi_j^p(x)}{|x - y|^{N+sp}} \, dxdyd\tau\nonumber\\
			\leq& \frac{\gamma 2^j}{\rho} \left[ \int_0^\theta \tau^{\frac{1}{p}} \iint_{A_{j,\tau}} \frac{(u(x, \tau) - u(y, \tau))^p}{|x - y|^{N+sp}} (u(x, \tau) + \kappa_j)^{-\frac{q+1}{p}} \xi_j^p(x) \, dxdyd\tau \right]^{\frac{p-1}{p}}\nonumber\\
			&\quad\quad\times\left[ \int_0^\theta \tau^{\frac{1-p}{p}} \iint_{A_{j,\tau}} \frac{(u(x, \tau) + \kappa_j)^{\frac{(q+1)(p-1)}{p}}}{|x - y|^{N+sp-p}} \xi_j^p(x) \, dxdyd\tau \right]^{\frac{1}{p}}. 
		\end{align}
		We now apply Lemma~\ref{lem-4-1} with $\rho=\rho_{j+1}$, $\sigma=\rho_j/\rho_{j+1}$,
		$\sigma'=\widehat\rho_j/\rho_{j+1}$ and $\xi=\xi_j$ to obtain
		\begin{align}
			\label{4.18}
			&\int_0^\theta \tau^{\frac{1}{p}} \iint_{A_{j,\tau}} \frac{(u(x, \tau) - u(y, \tau))^p}{|x - y|^{N+sp}} (u(x, \tau) + \kappa_j)^{-\frac{q+1}{p}} \xi_j^p(x) \, dxdyd\tau\nonumber\\
			\leq &\gamma b^j \rho^s \left( \frac{\theta}{\rho^{\lambda}} \right)^{\frac{1}{p}}\left[ S_{j+1} +\rho^N\left(\frac{\theta}{\rho^{sp}}\right)^{\frac{q}{q+1-p}}+\rho^N\mathcal{T}_-\left(u;x_0,\frac{\rho}{2};0,\theta\right) \right]^{\frac{(q+1)(p-1)}{pq}},
		\end{align}
		where $b>1$ is a constant depending only on $N,p,s,q$.  In addition, we have by H\"{o}lder's inequality that
		\begin{align}
			\label{4.19}
			&\int_0^\theta \tau^{\frac{1-p}{p}} \iint_{A_{j,\tau}} \frac{(u(x, \tau) + \kappa_j)^{\frac{(q+1)(p-1)}{p}}}{|x - y|^{N+sp-p}} \xi_j^p(x) \, dxdyd\tau\nonumber\\
			\leq& \gamma \theta^{\frac{1}{p}} \rho^{\frac{N(q+1-p)}{pq}+p-sp} \left[ \sup_{0 < \tau < \theta} \int_{K_{j+1}} u^q(x, \tau) dx + \kappa_j^q \rho_{j+1}^N \right]^{\frac{(q+1)(p-1)}{pq}}\nonumber\\
			\leq &\gamma \rho^{p+s-sp} \left( \frac{\theta}{\rho^{\lambda}} \right)^{\frac{1}{p}} \left[ S_{j+1} +\rho^N\left(\frac{\theta}{\rho^{sp}}\right)^{\frac{q}{q+1-p}}+\rho^N\mathcal{T}_-\left(u;x_0,\frac{\rho}{2};0,\theta\right) \right]^{\frac{(q+1)(p-1)}{pq}}.
		\end{align}
		Substituting \eqref{4.18} and \eqref{4.19} into \eqref{4.17} leads to
		\begin{align*}
			J_2\leq \gamma b^j \left( \frac{\theta}{\rho^{\lambda}} \right)^{\frac{1}{p}}\left[ S_{j+1} +\rho^N\left(\frac{\theta}{\rho^{sp}}\right)^{\frac{q}{q+1-p}}+\rho^N\mathcal{T}_-\left(u;x_0,\frac{\rho}{2};0,\theta\right) \right]^{\frac{(q+1)(p-1)}{pq}} .
		\end{align*}	
		Applying Young's inequality, we have 
		\begin{align*}
			J_2 \leq \varepsilon S_{j+1}+\gamma(\varepsilon)b^j\rho^N\left[\left(\frac{\theta}{\rho^{sp}}\right)^{\frac{q}{q+1-p}}+\mathcal{T}_-\left(u;x_0,\frac{\rho}{2};0,\theta\right)\right].
		\end{align*}   
		
		\textbf{Estimate of $J_3$:} Note that 
		\begin{align*}
			\frac{|y-x_0|}{|y-x|}\leq  1+\frac{|x-x_0|}{|y-x|}\leq \gamma 2^j
		\end{align*}
		for all $x \in \widetilde{K}_j $, $y \in \mathbb{R}^N\backslash\widehat{K}_j$. Thus we apply H\"{o}lder's inequality and Young's inequality to obtain
		\begin{align*}
			J_3\leq& \frac{\gamma   }{\rho^{sp}} \theta 2^{(N+sp)j}\left[ \sup_{0<\tau<\theta} \int_{\widetilde{K}_j} u^{p-1}(x, \tau) \,dx\right] \\
			\leq&\frac{\gamma b^j \theta}{\rho^{sp}} \left[ \sup_{0<\tau<\theta} \int_{K_{j+1}} u^q(x,\tau)\,dx \right]^{\frac{p-1}{q}} |K_{j+1}|^{\frac{q+1-p}{q}}\\
			\leq &\frac{\gamma b^j \theta}{\rho^{\lambda}} S_{j+1}^{\frac{p-1}{q}} \leq \varepsilon S_{j+1} + \gamma(\varepsilon) b^j \left( \frac{\theta}{\rho^{\lambda}} \right)^{\frac{q}{q+1-p}},
		\end{align*}
		where $b>1$ depends only on the \textit{data}.
		
		\textbf{Estimate of $J_4$:} We deduce that
		\begin{align*}
			J_4 &\leq \gamma \int_0^\theta \int_{\widetilde{K}_j } \int_{\mathbb{R}^N\backslash\widehat{K}_j} \frac{u_-^{p-1}(y, \tau)}{|x - y|^{N+sp}} \,dydx d\tau\\
			&\leq \gamma b^j \rho^N\mathcal{T}_{-}(u;x_0,\frac{\rho}{2};0,\theta),
		\end{align*}
		with $b>1$ depending only on the \textit{data}.
		
		We next consider the case $t_1>t_2$.
		Testing \eqref{1.1} with $-\xi_j^{p+1}$ over $(t_2,t_1)$,
		we obtain the same decomposition as in \eqref{4.16}, with the
		sign of the nonlocal interaction reversed. For the interior interaction, set
		$$B_{j,\tau}:=\left\{(x,y)\in\widehat K_j\times\widehat K_j: u(x,\tau)>u(y,\tau),\ \xi_j(x)<\xi_j(y)
		\right\}.$$
		By symmetry, it holds that
		\begin{align*}
			\widetilde J_2 \leq \gamma \int_0^\theta \iint_{B_{j,\tau}}\frac{(u(x,\tau)-u(y,\tau))^{p-1}}
			{|x-y|^{N+sp}}(\xi_j^{p+1}(y)-\xi_j^{p+1}(x))\,dxdyd\tau.
		\end{align*}
		Using
		$$\xi_j^{p+1}(y)-\xi_j^{p+1}(x)\leq\gamma|\xi_j(y)-\xi_j(x)|\xi_j^p(y),$$
		the estimate over $B_{j,\tau}$ provided by Lemma \ref{lem-4-1}, together with the argument in \eqref{4.17}--\eqref{4.19}, we obtain
		\begin{align*}
			\widetilde J_2\leq{}\varepsilon S_{j+1}+\gamma(\varepsilon)b^j\rho^N\left[\left(\frac{\theta}{\rho^{sp}}\right)^{\frac{q}{q+1-p}}+\mathcal{T}_-\left(u;x_0,\frac{\rho}{2};0,\theta\right)\right].
		\end{align*}
		The local part of the exterior interaction is estimated exactly as
		$J_3$. In addition, we have
		\begin{align*}
			\widetilde J_4&\leq\gamma \int_0^\theta
			\int_{\widetilde K_j}
			\int_{\mathbb R^N\setminus\widehat K_j}
			\frac{u_+^{p-1}(y,\tau)}{|x-y|^{N+sp}}\,dydxd\tau\\
			&\leq
			\gamma b^j\rho^N\mathcal{T}_+\left(u;x_0,\frac{\rho}{2};0,\theta\right).
		\end{align*}
		
		Combining the estimates in both cases $t_1\leq t_2$ and $t_1>t_2$,  there holds
		\begin{align*}
			S_j\leq{}&A+\varepsilon S_{j+1}+\gamma(\varepsilon)b^j\rho^N
			\left[\left(\frac{\theta}{\rho^{sp}}\right)^{\frac{q}{q+1-p}}+\mathcal T_++\mathcal T_-\right]\\
			\leq{}&\varepsilon S_{j+1}+\gamma(\varepsilon)b^j\left\{A+\rho^N\left[
			\left(\frac{\theta}{\rho^{sp}}\right)^{\frac{q}{q+1-p}}+\mathcal T_++\mathcal T_-\right]\right\},
		\end{align*}
		where
		\begin{align*}
			\mathcal T_\pm:=\mathcal T_\pm\left(u;x_0,\frac{\rho}{2};0,\theta\right).
		\end{align*}
		Iterating the above inequality gives that
		\begin{align*}
			S_0\leq{}&\varepsilon^j S_j+\gamma(\varepsilon)\sum_{i=0}^{j-1}(\varepsilon b)^i
			\left\{A+\rho^N\left[\left(\frac{\theta}{\rho^{sp}}\right)^{\frac{q}{q+1-p}}+\mathcal T_+
			+\mathcal T_-\right]\right\}.
		\end{align*}
		Then, by letting $\varepsilon=1/(2b)$, we get $\sum_{i=0}^{j-1}(\varepsilon b)^i\leq 2$, and meanwhile $\varepsilon^j S_j \rightarrow 0$ as $j \rightarrow \infty$. The proof is completed by recalling the definitions of $S_j$ and $A$.
	\end{proof}
	
	We are now ready to prove Theorem \ref{thm-1-7}.
\begin{proof}[Proof of Theorem \ref{thm-1-7}]
	By translation, assume that $(x_0,t_0)=(0,0)$. We first show that
	\begin{align}
		\label{4.20}
		\operatorname*{ess\,sup}_{Q_{\rho/2,\theta/2}}u
		&\leq\gamma\left(\frac{\rho^{sp}}{\theta}\right)^{\frac{N}{\lambda_q}}
		\left(
		\mint\!\!\mint_{Q_{\rho,\theta}}u^q\,dx\,d\tau\right)^{\frac{sp}{\lambda_q}}
		\nonumber\\
		&\quad+\gamma\left(\frac{\theta}{\rho^{sp}}\right)^{\frac{1}{q+1-p}}
		+\gamma\mathcal T_+(u;0,\rho/2;-\theta,0)^{\frac1q}.
	\end{align}
	For $q\geq1$, this follows directly from
	Theorem~\ref{thm-1-3} with $r=q$.
	To include $0<q<1$, we consider the interpolation argument. Set
	\begin{align*}
		M_{s_i}:=\operatorname*{ess\,sup}_{Q_{s_i\rho,s_i\theta}}u,
		\qquad \frac12\leq s_1<s_2\leq\frac34.
	\end{align*}
	Since $\lambda_{q+1}=\lambda_q+sp>sp$, the estimate
	\eqref{3.22} from the proof of
	Theorem~\ref{thm-1-3} can be applied with
	$r=q+1$.
	Apply it in $Q_{s_2\rho,s_2\theta}$ with parameter
	$s_1/s_2$. Using $u^{q+1}\leq M_{s_2}u^q$ in that cylinder gives
	\begin{align*}
		M_{s_1}
		&\leq\frac{\gamma}{(s_2-s_1)^\kappa}
		\left(\frac{\rho^{sp}}{\theta}\right)^{\frac{N}{\lambda_{q+1}}}
		M_{s_2}^{\frac{sp}{\lambda_{q+1}}}
		\left(\mint\!\!\mint_{Q_{\rho,\theta}}u^q\,dxd\tau\right)^{\frac{sp}{\lambda_{q+1}}}\\
		&\quad+\frac{\gamma}{(s_2-s_1)^\kappa}
		\left[\left(\frac{\theta}{\rho^{sp}}\right)^{\frac{1}{q+1-p}}+\left(\frac{\theta}{\rho^{sp}}\right)^{\frac1q}M_{s_2}^{\frac{p-1}{q}}+M_{s_2}^{1-q}\mathcal T_+(u;0,\rho/2;-\theta,0)\right],
	\end{align*}
	where $\kappa$ depends only on the \textit{data}.  Employing Young's inequality with 
	$(s_2-s_1)^{-\kappa}$ leads to
	\begin{align*}
		M_{s_1}\leq&\frac12M_{s_2}
		+\frac{\gamma}{(s_2-s_1)^{\kappa_1}}
		\bigg[\left(\frac{\rho^{sp}}{\theta}\right)^{\frac{N}{\lambda_q}}
		\left(\mint\!\!\mint_{Q_{\rho,\theta}}u^q\,dxd\tau\right)^{\frac{sp}{\lambda_q}}\\&
		+\left(\frac{\theta}{\rho^{sp}}\right)^{\frac{1}{q+1-p}}
		+\mathcal T_+^{\frac1q}(u;0,\rho/2;-\theta,0)\bigg].
	\end{align*}
	The iteration on $[1/2,3/4]$, together with the local boundedness of $u$ proves
	\eqref{4.20} also for $0<q<1$.
	
	By Theorem~\ref{thm-1-6}, we obtain
	\begin{align*}
		\mint\!\!\mint_{Q_{\rho,\theta}}u^q\,dx\,d\tau
		&\leq\frac{\gamma}{\rho^N}\operatorname*{inf}_{-\theta<\tau<0}
		\int_{K_{2\rho}}u^q(x,\tau)\,dx\\
		&\quad+\gamma\left(\frac{\theta}{\rho^{sp}}\right)^{\frac{q}{q+1-p}}
		+\gamma\mathcal T_+(u;0,\rho/2;-\theta,0)
		+\gamma\mathcal T_-(u;0,\rho/2;-\theta,0).
	\end{align*}
	Substituting this bound into \eqref{4.20} and using $\lambda_q=spq-N(q+1-p)$ gives
	\begin{align}
		\label{integrated-tail-intermediate}
		\operatorname*{ess\,sup}_{Q_{\rho/2,\theta/2}}u
		&\leq\gamma\theta^{-\frac{N}{\lambda_q}}
		\left[\operatorname*{inf}_{-\theta<\tau<0}
		\int_{K_{2\rho}}u^q(x,\tau)\,dx\right]^{\frac{sp}{\lambda_q}}+\gamma\left(\frac{\theta}{\rho^{sp}}\right)^{\frac{1}{q+1-p}}
		\nonumber\\
		&\quad+\gamma\left(\frac{\rho^{sp}}{\theta}\right)^{\frac{N}{\lambda_q}}
		\mathcal T_+(u;0,\rho/2;-\theta,0)^{\frac{sp}{\lambda_q}}
		\nonumber\\
		&\quad+\gamma\left(\frac{\rho^{sp}}{\theta}\right)^{\frac{N}{\lambda_q}}
		\mathcal T_-(u;0,\rho/2;-\theta,0)^{\frac{sp}{\lambda_q}}
		+\gamma\mathcal T_+(u;0,\rho/2;-\theta,0)^{\frac1q}.
	\end{align}
	Since $sp/\lambda_q>1/q$, we get the desired result.
\end{proof}

\section{Extinction time and decay rates}
\label{sec5}
The first part of this section concerns the homogeneous Cauchy--Dirichlet problem \eqref{1.3}.
We estimate the finite extinction time in terms of the initial datum and volume of the domain when necessary. We begin with the following auxiliary lemma.

\begin{lemma}
	\label{lem-5-1}
	Let $0<p-1<q$, $sp<N$, and $\lambda_{q+1}=N(p-q-1)+sp(q+1)$. Let $E$ be a bounded open set. Assume that the initial datum $u_0$ is nonnegative and satisfies $u_0\in W_{0}^{s,p}(E)\cap L^\infty(\mathbb{R}^N)$. Let $u$ be a nonnegative, bounded weak solution to \eqref{1.3}. Define 
	\begin{align*}
	\alpha=\frac{N(q+1-p)}{sp},\quad\text{ and   }  \quad p_\alpha=\frac{N(q+1)}{N+s(q+1)}.\end{align*}
	Then the following estimates hold:
\begin{itemize}
		\item[(i)] if $1 < p < p_\alpha$, then there exists a constant $\gamma(N,p,s,q,\Lambda)>0$ such that 
		\begin{align}
			\label{5.1}
		\int_0^T \left[ - \|u(\cdot, \tau)\|_{L^\alpha(E)}^\alpha \psi'(\tau) + \frac{1}{\gamma}\|u(\cdot, \tau)\|_{L^{\alpha}(E)}^{p+\alpha-q-1} \psi(\tau) \right] d\tau \leq 0,
		\end{align}
		for all $\psi\in W_0^{1,q+1}(0,T)$ satisfying $\psi\geq0$ in $(0,T)$;
		\item[(ii)] if $p_\alpha\leq p$, then there exists a constant $\gamma (N,p,s,q,\Lambda)> 0$ such that 
		\begin{align}
			\label{5.2}
		\int_0^T \left[ - \|u(\cdot, \tau)\|_{L^{q+1}(E)}^{q+1}\psi'(\tau) + \frac{1}{\gamma}|E|^{-\frac{\lambda_{q+1}}{(q+1)N}} \|u(\cdot, \tau)\|_{L^{q+1}(E)}^p \psi(\tau) \right] d\tau \leq 0,
		\end{align}
		for all $\psi \in W_0^{1,q+1}(0,T)$ satisfying $\psi \geq 0$ in $(0,T)$.
	\end{itemize}
\end{lemma}
\begin{proof}
	We first prove (i). Since $1<p<p_\alpha$, we have
	$\alpha>q+1$. Taking
	\begin{align*}
		\varphi_h(x,\tau)=[\![u]\!]_{\bar{h}}^{\alpha-q}(x,\tau)\psi(\tau)
		\end{align*}
	as a test function in the weak formulation of \eqref{1.3} and using the fact that
	\(\psi\in W^{1,q+1}_0(0,T)\), we obtain
	\begin{align}
		\label{5.3}
		0 = &- \int_0^T \int_E u^q(x, \tau) \frac{\partial \varphi_h}{\partial \tau}(x, \tau)\, dx d\tau\nonumber \\
		&\quad +\frac12 \int_0^T \int_{\mathbb{R}^N }\int_{\mathbb{R}^N} U(x,y,\tau) (\varphi_h(x, \tau) - \varphi_h(y, \tau)) K(x, y, \tau) dx dy d\tau.
		\end{align}
		We estimate the first term as 
		\begin{align*}
			&\quad- \int_0^{T} \int_{E}  u^q(x, \tau) \frac{\partial \varphi_h}{\partial \tau}(x, \tau) \, dx d\tau\\
			&=-\int_0^{T} \int_{E} [\![u]\!]_{\bar{h}}^q(x, \tau) \frac{\partial \varphi_h}{\partial \tau}(x, \tau) \, dxd\tau + \int_0^{T} \int_{E}  ([\![u]\!]_{\bar{h}}^q(x, \tau) - u^q(x, \tau)) \frac{\partial \varphi_h}{\partial \tau}(x, \tau) \, dxd\tau\\
			&=:G_1+G_2.
		\end{align*}
	By integration by parts, we have
		\begin{align*}
		G_1=& \frac{q}{\alpha} \int_0^{T}\int_{E} \psi(\tau)
		\partial_\tau [\![u]\!]_{\bar{h}}^\alpha(x,\tau)\,dxd\tau\\
		=& -\frac{q}{\alpha}\int_0^{T}\int_{E} \psi'(\tau)
		[\![u]\!]_{\bar{h}}^\alpha(x,\tau)\,dxd\tau.
	\end{align*}
	Then it follows from Lemma~\ref{lem-2-8}(i) and the boundedness of $u$ that
	\begin{align*}
		\lim_{h\to 0} G_1=-\frac{q}{\alpha}\int_0^{T}\int_{E} \psi'(\tau)
		u^\alpha(x,\tau)\,dxd\tau.
		\end{align*}
Using Lemma~\ref{lem-2-8}(ii), we compute
	\begin{align*}
			G_2 & =(\alpha-q) \int_0^T \int_{E}\left([\![u]\!]_{\bar{h}}^q(x, \tau)-u^q(x, \tau)\right) [\![u]\!]_{\bar{h}}^{\alpha-q-1}(x, \tau) \frac{\partial [\![u]\!]_{\bar{h}}}{\partial \tau}(x, \tau) \psi(\tau) \,d x d \tau \\
			& \quad+\int_0^T \int_{E}\left([\![u]\!]_{\bar{h}}^q(x, \tau)-u^q(x, \tau)\right)[\![u]\!]_{\bar{h}}^{\alpha-q}(x, \tau) \psi^{\prime}(\tau) \,d x d \tau \\
			& =(\alpha-q) \int_0^T \int_{E} \frac{\left([\![u]\!]_{\bar{h}}^q(x, \tau)-u^q(x, \tau)\right)\left([\![u]\!]_{\bar{h}}(x, \tau)-u(x, \tau)\right) }{h}[\![u]\!]_{\bar{h}}^{\alpha-q-1}(x, \tau)\psi(\tau)\, d x d \tau \\
			& \quad+\int_0^T \int_{E} [\![u]\!]_{\bar{h}}^{\alpha}(x, \tau) \psi^{\prime}(\tau) -u^q(x, \tau)[\![u]\!]_{\bar{h}}^{\alpha-q}(x, \tau) \psi^{\prime}(\tau)\,d x d \tau.
		\end{align*}
	Note that the first term in the above estimate is nonnegative, while the second tends to zero as $h\to0$. Thus,
	\begin{align*}
	\liminf_{h \to 0} G_2\geq 0. 
	\end{align*}
	Consequently,
	\begin{align}
		\label{5.4}
			\liminf_{h \to 0}-  \int_0^T \int_E u^q(x, \tau) \frac{\partial \varphi_h}{\partial \tau}(x, \tau)\, dx d\tau \geq -\frac{q}{\alpha} \int_0^T \|u(\cdot, \tau)\|_{L^\alpha(E)}^\alpha \psi'(\tau) \, d\tau.
	\end{align}
To pass to the limit in the nonlocal term, we use the strong convergence of the time averages in
$L^p(0,T;W_0^{s,p}(E))$ and their uniform boundedness. Since $\alpha-q>1$, the continuity of the map $r\mapsto r^{\alpha-q}$ yields
\begin{align*}
	\varphi_h\longrightarrow u^{\alpha-q}\psi
	\quad\text{strongly in }
	L^p(0,T;W_0^{s,p}(E)).
	\end{align*}
	By \eqref{kernel} and H\"older's inequality, we get
	\begin{align*}
	\lim_{h\to0}\int_0^T\mathcal E(u,\varphi_h,\tau)\,d\tau=
	\int_0^T\mathcal E(u,u^{\alpha-q}\psi,\tau)\,d\tau.
	\end{align*}
	Moreover,
	\begin{align*}
	U(x,y,\tau)
	\bigl(u^{\alpha-q}(x,\tau)-u^{\alpha-q}(y,\tau)\bigr)\geq0.
	\end{align*}
An application of the lower bound in \eqref{kernel}, Lemma~\ref{lem-2-7}, and the fractional Sobolev inequality \cite[Theorem~6.5]{DPV12} leads to
	\begin{align}
		\label{5.5}
		&\liminf_{h \to 0}  \frac12 \int_0^T\int_{\mathbb{R}^N }\int_{\mathbb{R}^N} U(x,y,\tau) (\varphi_h(x, \tau) - \varphi_h(y, \tau)) K(x, y, \tau) \,dx dy d\tau\nonumber \\
		\geq &\frac{1}{\gamma}\int_0^T \int_{\mathbb{R}^N }\int_{\mathbb{R}^N}  \frac{U(x,y,\tau)}{|x - y|^{N+ps}} (u^{\alpha-q}(x, \tau) - u^{\alpha-q}(y, \tau)) \psi(\tau) \, dx dy d\tau \nonumber\\
		\geq &\frac{1}{\gamma} \int_0^T \int_{\mathbb{R}^N }\int_{\mathbb{R}^N}  \left| u^{\frac{p+\alpha-q-1}{p}}(x, \tau) - u^{\frac{p+\alpha-q-1}{p}}(y, \tau) \right|^p \frac{\psi(\tau)}{|x - y|^{N+ps}}\,dxdy d\tau \nonumber\\
		\geq &\frac{1}{\gamma}\int_0^T \left[ \int_{E} u^{\frac{p+\alpha-q-1}{p} p_s^*}(x, \tau) \, dx \right]^{\frac{p}{p_s^*}} \psi(\tau) \, d\tau\nonumber\\
		= &\frac{1}{\gamma} \int_0^T \|u(\cdot, \tau)\|_{L^\alpha(E)}^{p+\alpha-q-1} \psi(\tau) \, d\tau,
		\end{align}
	where $p_{s}^*=Np/(N-sp)$, and we have employed the relation
	\begin{align*}
	\left( \frac{p+\alpha-q-1}{p} \right) p_s^* = \alpha.
	\end{align*}
   Combining \eqref{5.3}--\eqref{5.5}, we arrive at
   \begin{align*}
   	0 \geq -\frac{q}{\alpha} \int_0^T \|u(\cdot, \tau)\|_{L^\alpha(E)}^\alpha \psi'(\tau) \, d\tau + \frac{1}{\gamma} \int_0^T \|u(\cdot, \tau)\|_{L^\alpha(E)}^{p+\alpha-q-1} \psi(\tau) \, d\tau.
   	\end{align*}
   	This completes the proof of (i). 
   	
   	For (ii), use the same argument with the test exponent $\alpha-q$ replaced by $1$. Since $p\geq p_\alpha$ is equivalent to $q+1\leq p_s^*$, H\"older's inequality on the set $E$ gives
\begin{align*}
   		\|u\|_{L^{p_s^*}(E)}^p
   		&\geq |E|^{p/p_s^*-p/(q+1)}\|u\|_{L^{q+1}(E)}^p\\
   		&=|E|^{-\lambda_{q+1}/[N(q+1)]}\|u\|_{L^{q+1}(E)}^p.
   	\end{align*}
   Combining this with the fractional Sobolev inequality gives \eqref{5.2}.
\end{proof}

Next, we give the proof of Theorem \ref{thm-1-8}.

\begin{proof}[Proof of Theorem \ref{thm-1-8}]
	We first consider (i). Set
	\begin{align*}
	U(t):=\|u(\cdot,t)\|_{L^\alpha(E)}^\alpha.
	\end{align*}
	Since $u$ is bounded and continuous in $L^{q+1}(E)$, we have $U\in C([0,T])$. Then it follows from Lemma \ref{lem-5-1} that
	\begin{equation}
		\label{5.6}
		\begin{cases}
			U'(t)+\dfrac1\gamma
			U(t)^{\frac{p+\alpha-q-1}{\alpha}}\leq0
			&\text{in }(0,T),\\[4pt]
			U(0)=\|u_0\|_{L^\alpha(E)}^\alpha,
		\end{cases}
	\end{equation}
	where the differential inequality is understood in the
	distributional sense. Since $U$ is continuous, \eqref{5.6} implies that $U$ is
	nonincreasing. If $U(0)=0$, then $U\equiv0$.
	Otherwise, fix $t\in(0,T)$ such that $U(t)>0$. Then
	\begin{align*}
	U(\tau)\geq U(t)>0\qquad\text{for }0\leq\tau\leq t.
	\end{align*}
	Applying time regularization to \eqref{5.6} on compact subintervals of $(0,t)$ and then using the chain rule and passing to the limit, we obtain
	\begin{align*}
	U^{\frac{q+1-p}{\alpha}}(t)-U^{\frac{q+1-p}{\alpha}}(0)
	\leq-\frac{q+1-p}{\alpha\gamma}\,t.
	\end{align*}
	By monotonicity, $U$ remains zero once it reaches zero. Thus one can get
	\begin{align*}
	\|u(\cdot,t)\|_{L^\alpha(E)}^{q+1-p}
	\leq\left[\|u_0\|_{L^\alpha(E)}^{q+1-p}-\frac{t}{\gamma}\right]_+,
	\qquad 0\leq t\leq T,
	\end{align*}
	where $\gamma>0$ depends only on the \textit{data}. Therefore, $u(\cdot, t)=0$ almost everywhere in $E$ whenever
	$$t \geq \gamma \|u_0\|_{L^\alpha(E)}^{q+1-p}.$$
Since $T>0$ is arbitrary and $u$ vanishes outside $E$, assertion (i) follows.

 For the case (ii), set
 \begin{align*}
 U(t):=\|u(\cdot,t)\|_{L^{q+1}(E)}^{q+1}.
 \end{align*}
 Then $U\in C([0,T])$, and Lemma~\ref{lem-5-1}(ii) gives
 \begin{equation}
 	\label{5.7}
 	\begin{cases}
 		U'(t)+\dfrac1\gamma
 		|E|^{-\frac{\lambda_{q+1}}{N(q+1)}}
 		U(t)^{\frac{p}{q+1}}\leq0
 		&\text{in }(0,T),\\[4pt]
 		U(0)=\|u_0\|_{L^{q+1}(E)}^{q+1},
 	\end{cases}
 \end{equation}
 where the differential inequality is understood in the
 distributional sense. Since $p<q+1$, the same argument as in (i) yields
 \begin{align*}
 \|u(\cdot,t)\|_{L^{q+1}(E)}^{q+1-p}\leq
 \left[ \|u_0\|_{L^{q+1}(E)}^{q+1-p}-\frac{t}{\gamma} |E|^{-\frac{\lambda_{q+1}}{N(q+1)}}
 \right]_+, \qquad 0\leq t\leq T.
 \end{align*}
 Since $T>0$ is arbitrary and $u$ vanishes outside $E$, assertion (ii) follows.
\end{proof}	

Finally, we derive upper bounds describing the decay of $u$ as $t$ approaches the extinction time $T_*$.
\begin{proof}[Proof of Theorem \ref{thm-1-9}]
We first prove (i). Fix $0<t<\tau<T_*$. Applying the one-sided integral Harnack inequality
	\eqref{1.9} and using the global nonnegativity of $u$, we can deduce 
	\begin{align*}
		\sup_{t<s<\tau}\int_{K_\rho(x_0)}u^q(x,s)\,dx
		\leq
		\gamma\int_{K_{2\rho}(x_0)}u^q(x,\tau)\,dx+\gamma\rho^N
		\left(\frac{\tau-t}{\rho^{sp}}\right)^{\frac{q}{q+1-p}}.
	\end{align*}
	By the local $L^{q+1}$ time continuity of $u$ and $u(\cdot,T_*)=0$, we have
	\begin{align*}
		\lim_{\tau\rightarrow T_*}
		\int_{K_{2\rho}(x_0)}u^q(x,\tau)\,dx=0.
	\end{align*}
	Consequently, we let $\tau\rightarrow T_*$ to obtain 
		\begin{equation}\label{eq:extinction-mass-uniform}
		\sup_{t<s<T_*}
		\int_{K_\rho(x_0)}u^q(x,s)\,dx
		\leq
		\gamma\rho^N
		\left(\frac{T_*-t}{\rho^{sp}}\right)^
		{\frac{q}{q+1-p}}.
	\end{equation}
	By time continuity, the same bound holds at $s=t$.
	
	Next, we estimate the tail of the positive part of $u$. For the fixed $t\in(0,T_*)$, set
	\begin{align*}
		\kappa:=\left(\frac{T_*-t}{\rho^{sp}}\right)^{\frac{1}{q+1-p}}.
	\end{align*}
	Choose $\eta\in C_c^\infty(K_{\rho/4}(x_0))$
	such that
	\begin{align*}
		0\leq\eta\leq1,\qquad
		\eta=1\ \text{in }K_{\rho/8}(x_0),
		\qquad
		|D\eta|\leq\frac{\gamma}{\rho}.
	\end{align*}
	Since $u(\cdot,T_*)=0$, testing the equation with $\eta^{p+1}$ on $(t,T_*)$ yields
	\begin{equation}\label{eq:extinction-flux-balance}
		\int_t^{T_*}G(\tau)\,d\tau
		=\int_{K_{\rho/2}(x_0)}
		u^q(x,t)\eta(x)^{p+1}\,dx-\int_t^{T_*}F(\tau)\,d\tau,
	\end{equation}
	where
		\begin{align*}
		F(\tau)
		&:=\frac12\iint_{K_{\rho/2}(x_0)\times K_{\rho/2}(x_0)}K(x,y,\tau)|u(x,\tau)-u(y,\tau)|^{p-2}\\
		&\qquad\qquad\qquad\qquad\qquad\quad\times
		(u(x,\tau)-u(y,\tau))\bigl(\eta(x)^{p+1}-\eta(y)^{p+1}\bigr)\,dxdy,
	\end{align*}
	\begin{align*}
		G(\tau)
		&:=\int_{K_{\rho/2}(x_0)}\int_{\mathbb R^N\setminus K_{\rho/2}(x_0)}K(x,y,\tau)
		|u(x,\tau)-u(y,\tau)|^{p-2}\\
		&\qquad\qquad\qquad\qquad\qquad\quad\times
		(u(x,\tau)-u(y,\tau))\eta(x)^{p+1}\,dydx.
	\end{align*}

	We first estimate the term $F(\tau)$.
	Define
	\begin{align*}
		A_\tau&:=
		\bigl\{(x,y)\in K_{\rho/2}(x_0)\times K_{\rho/2}(x_0): u(x,\tau)>u(y,\tau),\ \eta(x)>\eta(y)\bigr\},\\
		B_\tau&:=
		\bigl\{(x,y)\in K_{\rho/2}(x_0)\times K_{\rho/2}(x_0): u(x,\tau)>u(y,\tau),\ \eta(x)<\eta(y)\bigr\}.
	\end{align*}
Using symmetry in $x$ and $y$, the kernel bound
\eqref{kernel}, and the estimate
\begin{align*}
	|\eta(x)^{p+1}-\eta(y)^{p+1}|
	\leq\frac{\gamma}{\rho}\max\{\eta(x),\eta(y)\}^p|x-y|,
\end{align*}
we obtain by H\"older's inequality
\begin{align*}
	\int_t^{T_*}|F(\tau)|\,d\tau
	\leq\frac{\gamma}{\rho}\mathcal A^{\frac{p-1}{p}}\mathcal B^{\frac1p},
\end{align*}
	where
	\begin{align*}
		\mathcal A
		&:=
		\int_t^{T_*}(\tau-t)^{\frac1p}\iint_{A_\tau\cup B_\tau}
		\frac{(u(x,\tau)-u(y,\tau))^p\max\{\eta(x),\eta(y)\}^p}{
		|x-y|^{N+sp}(u(x,\tau)+\kappa)^{\frac{q+1}{p}}}\,dxdyd\tau,\\
		\mathcal B
		&:=
		\int_t^{T_*}(\tau-t)^{\frac{1-p}{p}}\iint_{A_\tau\cup B_\tau}
		\frac{(u(x,\tau)+\kappa)^{\frac{(q+1)(p-1)}p}\max\{\eta(x),\eta(y)\}^p}{|x-y|^{N+sp-p}}\,dxdyd\tau.
	\end{align*}
After translating the time interval to $(0,T_*-t)$, we apply Lemma~\ref{lem-4-1} with radius $\rho$,
$\theta=T_*-t$, $\sigma=1/8$, and cutoff $\eta$. Indeed, $\eta=1$ on $K_{\rho/8}(x_0)$,
$\operatorname{supp}\eta\Subset K_{\rho/4}(x_0)
\subset K_{9\rho/16}(x_0)$, and $K_{\rho/2}(x_0)\subset K_{9\rho/16}(x_0)$. Moreover, $\max\{\eta(x),\eta(y)\}$ equals
$\eta(x)$ on $A_\tau$ and $\eta(y)$ on $B_\tau$ as required. Since $u\geq 0$ in $\mathbb{R}^N$, Lemma~\ref{lem-4-1} together with \eqref{eq:extinction-mass-uniform} yields
	\begin{align*}
		\mathcal A
		&\leq
		\gamma\rho^s\left(\frac{T_*-t}{\rho^{sp-\frac{N(q+1-p)}q}}\right)^{\frac1p}
		\left[\sup_{t<\tau<T_*}\int_{K_\rho(x_0)}u^q(x,\tau)\,dx+\rho^N\kappa^q
		\right]^{\frac{(q+1)(p-1)}{pq}}\\
		&\leq
		\gamma\rho^{N+s}\left(\frac{T_*-t}{\rho^{sp}}\right)^{\frac{q}{q+1-p}}.
	\end{align*}
	To estimate $\mathcal B$, we use $0\leq\eta\leq1$, integrate with respect to $y$, and apply
	H\"older's inequality in $x$ to obtain
	\begin{align*}
		\mathcal B
		&\leq
		\gamma\rho^{p-sp}\int_t^{T_*}(\tau-t)^{\frac{1-p}{p}}
		\int_{K_{\rho/2}(x_0)}(u(x,\tau)+\kappa)^{\frac{(q+1)(p-1)}p}\,dxd\tau\\
		&\leq
		\gamma\rho^{p-sp}(T_*-t)^{\frac1p}\rho^{\frac{N(q+1-p)}{pq}}
		\left[\sup_{t<\tau<T_*}\int_{K_\rho(x_0)}u^q(x,\tau)\,dx
		+\rho^N\kappa^q\right]^{\frac{(q+1)(p-1)}{pq}}\\
		&\leq
		\gamma\rho^{N+p+s-sp}\left(\frac{T_*-t}{\rho^{sp}}\right)^{\frac{q}{q+1-p}},
	\end{align*}
	where we used the relation
	\begin{align*}
	(q+1-p)+(q+1)(p-1)=pq.
	\end{align*}
	Combining the bounds for
	$\mathcal A$ and $\mathcal B$, we conclude that
	\begin{equation}\label{eq:extinction-interior-flux}
	\int_t^{T_*}|F(\tau)|\,d\tau
	\leq \gamma\rho^N\left(\frac{T_*-t}{\rho^{sp}}\right)^{\frac{q}{q+1-p}},
	\end{equation}
	where $\gamma$ depends only on the $\textit{data}$.
	
	We next estimate the term $G(\tau)$.
	For $x\in\operatorname{supp}\eta$, we know $\operatorname{dist}(\operatorname{supp}\eta,\,
	\mathbb R^N\setminus K_{\rho/2}(x_0))\geq \gamma\rho$. Moreover, the kernel condition \eqref{kernel} and $\eta=1$ on $K_{\rho/8}(x_0)$
	imply
	\begin{align*}
		\int_{K_{\rho/2}(x_0)}\eta(x)^{p+1}K(x,y,\tau)\,dx
		\geq\frac{\gamma\rho^N}{|y-x_0|^{N+sp}},
	    \qquad y\notin K_{\rho/2}(x_0)
	\end{align*}
	with $\gamma$ depending only on the \textit{data}.
	Therefore, we get
	\begin{align}
		\label{upper bound}
		\int_t^{T_*}G(\tau)\,d\tau
		&\leq
		\gamma\rho^{-sp}\int_t^{T_*}\int_{K_{\rho/2}(x_0)}
		u^{p-1}(x,\tau)\,dx\,d\tau-
		\gamma\rho^N
		\mathcal T_+
		(u;x_0,\rho/2;t,T_*)\nonumber\\
		&\leq
		\gamma\rho^N
		\left(\frac{T_*-t}{\rho^{sp}}\right)^
		{\frac{q}{q+1-p}}-\gamma\rho^N
		\mathcal T_+
		(u;x_0,\rho/2;t,T_*).
	\end{align}
	In the last step, we utilized
	\begin{align*}
		&\rho^{-sp}\int_t^{T_*}\int_{K_{\rho/2}(x_0)}
		u^{p-1}(x,\tau)\,dxd\tau\\
	   \leq &
		\gamma\rho^{-sp}(T_*-t)
		\rho^{N(1-\frac{p-1}{q})}
		\left[
		\sup_{t<\tau<T_*}
		\int_{K_\rho(x_0)}
		u^q(x,\tau)\,dx
		\right]^{\frac{p-1}{q}}\\
		\leq&
		\gamma\rho^N
		\left(\frac{T_*-t}{\rho^{sp}}\right)^
		{\frac{q}{q+1-p}}.
	\end{align*}
	Combining the estimates \eqref{eq:extinction-flux-balance}-\eqref{upper bound},
	we arrive at
	\begin{align*}
		&\gamma\rho^N \mathcal T_+(u;x_0,\rho/2;t,T_*)\\
	    \leq &
		\gamma\rho^N
		\left(\frac{T_*-t}{\rho^{sp}}\right)^
		{\frac{q}{q+1-p}}
		+\int_t^{T_*}F(\tau)\,d\tau-\int_{K_{\rho/2}(x_0)}u^q(x,t)\eta(x)^{p+1}\,dx\\
		\leq&\gamma\rho^N
		\left(\frac{T_*-t}{\rho^{sp}}\right)^{\frac{q}{q+1-p}}
		+\int_t^{T_*}|F(\tau)|\,d\tau\\
	    \leq&
		\gamma\rho^N\left(\frac{T_*-t}{\rho^{sp}}\right)^{\frac{q}{q+1-p}},
	\end{align*}
	which implies that
	\begin{equation}\label{eq:extinction-tail-bound}
		\mathcal T_+
		(u;x_0,\rho/2;t,T_*)
		\leq
		\gamma
		\left(\frac{T_*-t}{\rho^{sp}}\right)^
		{\frac{q}{q+1-p}}.
	\end{equation}
	Multiplying \eqref{eq:extinction-tail-bound} by $\rho^N$ and adding \eqref{eq:extinction-mass-uniform}, we conclude that
	\begin{align*}
		&\int_{K_\rho(x_0)}u^q(x,t)\,dx
		+\rho^N\mathcal T_+(u;x_0,\rho/2;t,T_*)\\
	   \leq&
		\gamma\rho^N\left(\frac{T_*-t}{\rho^{sp}}\right)^{\frac{q}{q+1-p}}=
		\gamma\left[\frac{(T_*-t)^q}{\rho^{\lambda_q}}\right]^{\frac{1}{q+1-p}}.
	\end{align*}
	This proves (i).
	
	We next prove (ii).
	Now assume $\lambda_q>0$. Local $L^{q+1}$ continuity and $u(\cdot,T_*)=0$ imply
	\begin{align*}
		\inf_{t<\tau<T_*}\int_{K_{2\rho}(x_0)}u^q(x,\tau)\,dx=0.
	\end{align*}
	Applying Theorem~\ref{thm-1-7} on $(t,T_*)$ and using $u\geq 0$ in $\mathbb{R}^N$ together with \eqref{eq:extinction-tail-bound}, we deduce
	\begin{align*}
		\operatorname*{ess\,sup}_{K_{\rho/2}(x_0)\times((t+T_*)/2,T_*)}u
		\leq\gamma\left(\frac{T_*-t}{\rho^{sp}}\right)^{\frac{1}{q+1-p}}.
	\end{align*}
	The same bound holds on the time slice $(t+T_*)/2$ by local
	$L^{q+1}$ continuity. This proves (ii).
\end{proof}

	\subsection*{Acknowledgments}
	This work was supported by the National Natural Science Foundation of China (No. 12471128 and 12601365) and the Postdoctoral Fellowship Program of the China Postdoctoral Science Foundation (No. GZC20261660).

	\subsection*{Conflict of interest} 
	The authors declare that there is no conflict of interest. 
	
	\subsection*{Data availability}
	No data was used for the research described in the article.
	
	\subsection*{AI disclosure} A preliminary proof of the local boundedness result in Section \ref{sec3} was developed independently and subsequently refined with the assistance of GPT-6.
	The proofs in Sections \ref{sec4} and \ref{sec5} were developed independently, with AI assistance limited to improvements in presentation. All computations, proofs, references, and conclusions were independently verified by the authors, who take full responsibility for the content of this paper.

\end{document}